\documentclass[9pt]{amsart}

\usepackage{graphicx}

\usepackage{hyperref}

\def\R{{\mathbb R}}

\def\N{{\mathbb N}}

\def\Z{{\mathbb Z}}
\def\1{{1\!\!\!1}}
\def\Id{{\mathbb I}}

\def\E{{\mathbb E}}
\def\eps{{\epsilon}}
\def\P{{\mathbb P}}

\def\cal{\mathcal}

\def\dist{{\rm{dist}}}

\def\eps{\varepsilon}

\newcommand{\be}{\begin{equation}}
\newcommand{\ee}{\end{equation}}
\numberwithin{equation}{section}

\newtheorem{theorem}{Theorem}
\newtheorem{prop}{Proposition}[section]
\newtheorem{cor}{Corollary}[section]
\newtheorem{defi}{Definition}[section]
\newtheorem{lemma}{Lemma}[section]

\title{Martin boundary of a killed non-centered random walk in a general cone.}
\author{Irina Ignatiouk-Robert}
\address{
{Universit\'e de Cergy-Pontoise,}
{D\'epartement de math\'ematiques,}
{2, Avenue Adolphe Chauvin,}
{95302 Cergy-Pontoise Cedex,}
{France}}
\date{\today}
\email{Irina.Ignatiouk@u-cergy.fr}
\keywords{Harmonic function, random walk, exit time, renewal function} 
\subjclass{60J45, 31C05, 60J10, 60J50}

\date{Received: date / Accepted: date}

\begin{document}
\begin{abstract} We investigate  Martin boundary  for a non-centered random walk on $\Z^d$ killed up on the time $\tau_\vartheta$ of the first exit from a convex cone  with a vertex at $0$.  The approach combines  large deviation estimates, the ratio limite theorem and the ladder height process. The results are applied to identify the Martin boundary  for a  random walk killed upon the first exit from  a convex cone having $C^1$ boundary. 
\end{abstract}
\maketitle

\section{Introduction.}

Before formulating our results we recall the  definition of the Martin boundary and the main classical results of this domain. 

Consider a transient irreducible sub-stochastic Markov chain $(Z(t))$ on a countable state space $E\subset \Z^d$ with transition probabilities $p(x,y), x,y\in E$. The Green function $G(x,y)$  and the Martin kernel $K(x,y)$ associated with the Markov chain $(Z(n))$ rae defined respectively by 
\[
G(x,y) ~=~ \sum_{n=0}^\infty \P_x(Z(n) = y)  \quad x,y\in E,
\]
and 
\[
K(x,y) ~=~ \frac{G(x,y)}{G(z_0,y)}, \quad x,y\in E,
\]
where $\P_x$ denotes the probability measure on the set of trajectories of $(Z(t))$ corresponding to the initial state $Z(0)=x$ and  $z_0$ is a given reference point in $E$. For irreducible Markov chains, the family of functions $(K(\cdot, y), \; y\in E)$ is relatively compact with respect to the topologie of point-wise convergence~: for any sequence of points $(y_n)\in E^\N$, there is a subsequence $(y_{n_k})$ for which the sequence of functions $K(\cdot, y_{n_k})$ converges point-wise on $E$. The Martin  compactification $E_M$ is defined as the unique  smallest compactification of
the  set $E$ for  which the  Martin kernels $K(z,\cdot)$  extend continuously~: a sequence $z_n\in E$  converges to a
point on a Martin  boundary $\partial_M E = E_M\setminus E$ of $E$ if it leaves  every finite subset on $E$ and  the sequence  of functions  $K(\cdot,z_n)$  converges point-wise.  

Recall that a function $h : E\to\R_+$ is harmonic for $(Z(t))$
if $
\E_z(h(Z(1))) ~=~ h(z)$ 
for all $z\in E$. By the Poisson-Martin
representation  theorem, for  every  non-negative  harmonic function  $h$  there exists  a
positive Borel measure $\nu$ on $\partial_M E$ such that
\[
h(z) = \int_{\partial_M E} K(z,\eta) \, d\nu(\eta) 
\]
By Convergence theorem, the sequence
$(Z(t))$ converges $\P_z$ almost surely for every $z\in E$ to a $\partial_M E$ valued
random variable. The Martin boundary provides therefore the non-negative harmonic functions and 
shows how the Markov chain $(Z(t))$ goes to infinity.  

The concept of the Martin compactification was introduced  for a countable Markov chain by Doob~\cite{Doob} based on the ideas of Martin~\cite{Martin} regarding harmonic functions  on Euclidean domains for Brownian motion.

To identify the Martin boundary on has to investigate all possible  limites of the Martin kernel $K(x,y_n)$ when $\|y_n\|\to\infty$. A large number of results in this domain has been obtained for homogeneous random walks. 
Classical results are those of Dynkin and Malyutov~\cite{Dynkin_Malyutov:1961} for a random walk on free groupes and Ney and Spietzer~\cite{Ney-Spitzer} for a random walk on $\Z^d$. For a wide literature of results where the Martin boundary was identifies for more general homogeneous Markov chains we refer to the book of Woess~\cite{Woess} and the references therein.

For non-homogeneous Markov chains, the problem of the explicit description of the Martin compactification is usually a highly non-trivial task, and up to now there are few results where the Martin boundary for a non-homogeneous Markov chain was was identified explicitly. For random walks on non-homogeneous trees the Martin boundary was described by Cartier~\cite{Cartier}.  Alili and  Doney~\cite{Alili-Doney} identified the Martin boundary for space-time random  walk $S(n)=(Z(n),n)$ for a  homogeneous random walk $Z(n)$ on $\Z$ killed  when hotting  the negative half-line
$\{z : z<0\}$. These results were obtained by using the one dimensional structure of the process.  Doob~\cite{Doob}  identified   the Martin boundary for Brownian motion on a half-space by using an explicit form of the Green function. 
A method of complex  analysis on the elliptic curves  was proposed by Kurkova
and  Malyshev~\cite{Kurkova-Malyshev} to identify the Martin boundary for nearest neighbor random walks with drift in $\N\times\Z$ and $\Z_+^2$.  In the papers Raschel~\cite{KilianRaschel:09}, Raschel~\cite{KilianRaschel:10}, Kurkova and Raschel~\cite{Kurkova-Raschel:11}  this method  was developed to investigate the exact asymptotics of the Green function and identify the  Martin boundary for  random walks in  $\Z_+^2$ having small steps and an absorption condition on the  boundary. Because of the use  of the  specific  elliptic  curves,  these methods  seem to  be difficult to apply for higher dimensions.

In   order  to  identify  the  Martin boundary of a partially  homogeneous random walk on a  half-space $\Z^{d-1}\times\N$, a large deviations approach combined with the method of  Choquet-Deny theory  and the ratio limit theorem of Markov-additive processes was proposed by Ignatiouk-Robert~\cite{Ignatiouk-half_space_2008,  Ignatiouk:2010}  .  The method of Choquet-Deny
theory was used there in order to identify the minimal harmonic functions. Next the  limiting behavior of the  Martin kernel was investigated by using  the
large  deviation   estimates  of  the Green  function  and a  ratio   limit  theorem. It should be mentioned that the methods of Choquet-Deny theory and the ratio
limit  theorem obtained in \cite{Ignatiouk-half_space_2008}  are  valid  only  for Markov-additive  processes,  i.e.   when  transition probabilities are invariant  with respect to the translations on  some directions.

In order to identify the Martin boundary for a non-centered random walk in $\Z^2$ killed upon the first exit from the positive quadrant $\Z_+^2$, Ignatiouk-Robert and Lorée~\cite{Ignatiouk-Loree} applied the methods of Ignatiouk-Robert~\cite{Ignatiouk-half_space_2008,  Ignatiouk:2010} for local Markov-additive processes, obtained from the original killed random walk in $\Z_+^2$ by removing one of the boundaries $\{(x,y)\in\Z^2~:  x =0\}$ or $\{(x,y)\in\Z^2~:  y =0\}$.
The Martin boundary of the original killed random walk was obtained next by using the large deviation estimates and the functional equations relating the Green function of the original random walk with the Green functions of the local random walks. Unfortunately, this method can be applied only in the case when the local processes, obtained by removing on of the boundaries $\{(x,y)\in\Z^2~:  x =0\}$ or $\{(x,y)\in\Z^2~:  y =0\}$, are Markov-additive.

For centered random walks in $\Z^d$, killed upon the first exit from some cone ${\cal C}\subset\R^d$, under various assumptions on the cone, the Martin boundary was identified  by Duraj and Wachtel in ~\cite{Duraj-Wachtel} and by Jetlir Duraj, Kilian Raschel, Pierre Tarrago and Vitali Wachtel in ~\cite{Duraj-Raschel-Tarrago-Wachtel}. These results use the method of the diffusion approximation of the centered random walks and prove that under some general assumptions, for centered random walks killed upon the first exit from a convex cone, the Martin boundary is reduced to one point.  

To my knowledge, for a non-centered random walk on $\Z^d$ killed upon the first exit from a general convex cone, the only related existing results  are those of Rodolphe Garbit, Kilian Raschel~\cite{Garbit-Raschel}  and  Duraj~\cite{JDuraj}. Garbit and Raschel~\cite{Garbit-Raschel}  investigated the exponential decay of the probability that a given multi-dimensional random walk stays in a convex cone up to time $n$ as $n\to\infty$. Duraj~\cite{JDuraj} proved  the existence of uncountably many nonnegative harmonic functions and obtained an explicit representation of some of them.  These results do not allow to identify the Martin boundary of the process. 

In the present paper, we investigate the Martin boundary for a non-centered random walk on $\Z^d$ killed upon the first exit from a convex cone with a vertex at $0$. A simple exemple of such a random walk is a random walk in a half-plane 
$\{x\in\Z^2~:~ x\cdot \gamma \geq 0\}$,  where $x\cdot \gamma$ denotes a usual scalar product in $\R^2$ of the vectors $x$ and $\gamma$, for an arbitrary non-zero vector $\gamma =(\gamma_1,\gamma_2)\in\R^2$. If the real number $\gamma_1/\gamma_2$ is irrational, such a random walk is not Markov-additive and the methods developed in the papers \cite{Ignatiouk-half_space_2008,  Ignatiouk:2010, Ignatiouk-Loree}  do not work.  

The main ideas of our approach are the following.

i) Our first result improves the ratio limit theorem of Ignatiouk-Robert~\cite{Ignatiouk-half_space_2008,  Ignatiouk:2010}. This result proves that under some general assumptions, 
whenever $u+ E\subset E$, 
\be\label{int_eq1} 
\P_{x+u}(Z(1)=y+u) ~\geq~ \P_x(Z(1)=t), \quad \forall x,y\in E. 
\ee
and 
\be\label{int_eq2}
\liminf_{n\to\infty} \frac{1}{\|y_n\|}\log G(0,y_n) \geq 0,
\ee
the following relations hold 
\[
\liminf_n ~\frac{G(z+u,y_n)}{G(z,y_n)} ~\geq~ 1, \quad \forall x,y\in E. 
\]

ii) As a straightforward consequence of our ratio limit theorem, we obtain the following property: if \eqref{int_eq1} and \eqref{int_eq2} hold and the sequence $(y_n)\in E^\N$ converges in the Martin compactification to some point $\eta\in\partial_M E$ then the limit function $K(x,\eta) = \lim_n K(x,y_n)$ satisfies the following relations 
\be\label{int_eq3}
K(x+u, \eta) \geq K(x, \eta), \quad \forall x\in E. 
\ee

iii) Next we consider the case when  the set $E$ is an intersection of $\Z^d$ with a closed convex cone ${\cal C}$ having a vertex at $0$, and  relations \eqref{int_eq1} hold for all $x,y,u\in E$. Our next result proves that whenever 
the time of the first exit of the process $(Z(t))$ from the cone ${\cal C}$  is non integrable, any harmonic function satisfying the inequalities \eqref{int_eq3} is proportional to the renewal function $V$ of the corresponding ladder height process. 

iv) With these results, for a homogeneous random walk $(Z(t))$ on $\Z^d$ with a non-zero mean step $m$, when the time of the first exit of the process $(Z(t))$ from $E$ is non-integrable, we are able to prove that 
\[
\lim_n \frac{G(x,y_n)}{G(0,y_n)}  = V(x), \quad \forall x\in E, 
\]
for any sequence  $(y_n)\in E^\N$ with $\lim_n \|y_n\| = + \infty$ and $\lim_n y_n/\|y_n\| = m/\|m\|$.  Relation \eqref{int_eq2} follows in this cas from the large deviation estimates of the transition probabilities of $(Z(t))$.

v) To investigate the limites 
\be\label{int_eq4}
\lim_n \frac{G(x,y_n)}{G(0,y_n)} \quad \text{as \; $\lim_n \|y_n\| = + \infty$ \; and \; $\lim_n y_n/\|y_n\| = q\in {\cal C}$,}
\ee 
when $q\not= m/\|m\|$, the method of the exponential chain of measure is applied. With our approach we are able to identify these limites for those direction $q$ for which  the time of the first exit from ${\cal C}$ of the corresponding twisted process $(Z_q(t))$ is non-integrable. 

vi) Our last result proves that under some general assumptions, in the case when a convex cone ${\cal C}$ has a $C^1$-boundary, the first time, when a homogeneous random walk $(Z(t))$ with a non-zero mean step $q\in{\cal C}$ exists from the cone ${\cal C}$, is always non-integrable. When combined with our previous results, this result allow to identify the limites \eqref{int_eq4} in this particular case for any direction $q\in{\cal C}$.

\section{Main Results} 

Our first result improves the ratio limit theorem of Ignatiouk-Robert~\cite{Ignatiouk-half_space_2008}.  We assume here that 

\begin{enumerate} 
\item[(A1)]{\em The Markov chain $(Z(t))$ is irreducible on $E\subset \Z^d$.}
\item[(A2)] {\em There are $C>0$ and $\delta > 0$ such that 
\[
\sup_{\alpha\in\R^d~: \|\alpha\| \leq \delta_0} ~\sup_{x\in E} \E_x\Bigl(\exp\bigl(\alpha \cdot (Z(1) - x)\bigr)\Bigr) ~<~C,
\]
where $\alpha\cdot x$ denotes the usual scalar product in $\R^d$. }
\item[(A3)] {\em There is $u\in\Z^d\setminus\{0\}$ such that $E + u \subset E$ and 
\be\label{main_result_e1}
p(x+u, y+u) ~\geq~ p(x,y), \quad \forall x,y \in E,
\ee
}
\item[(A4)] {\em For some $\tilde{n}\geq 1$, 
\be\label{main_result_e2} 
\inf_{x\in E} \P_x(Z(\tilde{n}) = x+u) ~>~ 0.
\ee
}
\end{enumerate} 
\noindent
The Markov chain $(Z(t))$ being sub-stochastic, it is convenient to introduce an additional absorbing state $\vartheta$ by letting 
\[
p(x,\vartheta) ~=~1 - \sum_{y\in E} p(x,y) \quad \text{and} \quad p(\vartheta,\vartheta) = 1. 
\]

\begin{theorem}\label{ratio_limit_theorem} Suppose that the conditions (A1)-(A4) are satisfied and  let a sequence  $(y_n)\in E^\N$  with $\lim_n\|y_n\| = +\infty$ satisfy the  inequalities 
\be\label{main_result_e3}
\liminf_{n\to\infty} \frac{1}{\|y_n\|}\log G(0,y_n) \geq 0. 
\ee
Then for any $z\in E$, 
\be\label{main_result_e4} 
\liminf_n ~\frac{G(z+u,y_n)}{G(z,y_n)} ~\geq~ 1
\ee
\end{theorem} 

This result is an analogue of the ratio limit theorem obtained for Markov-additive processes in the paper \cite{Ignatiouk-half_space_2008}. The main idea of the proof of this theorem is the following~: 
because of \eqref{main_result_e3},  for any $z\in E$, the terms of the order $c\exp(-\delta \|y_n\|)$ give an
asymptotically negligible contribution to $G(z,y_n)$. To get \eqref{main_result_e4}, we decompose the quantities $G(z,y_n)$  into a negligible part
of the order $c\exp(-\delta \|y_n\|)$ and a main part $\Xi_\sigma(z,y_n)$. The main part is newt compared to $G(z+u,y_n)$ by using the Bernoulli part decomposition method.

The method of the Bernoulli part decomposition was initially proposed by Foley and McDonald~\cite{Foley-McDonald} for random walks in a half-plane. In the paper \cite{Ignatiouk-half_space_2008}, it was extended  for general Markov-additive processes, when the state space $E$ and the transition probabilities $p(x,y), x,y\in E$, of the process $(Z(t))$ are invariant with respect to the shifts on the vector $u$, e.i. when $E + u = E$ and $p(x+u, y+u) ~=~ p(x-u,y-u) = p(x,y)$ for all $x,y\in E$. In our setting, these relations are replaced respectively by $E+u \subset E$ and $p(x+u,y+u) \geq p(x,y)$. This is the main difficulty of our proof. To adapte the method of Bernoulli part decomposition to our case we combine this method with the method of coupling. 

\noindent 
The proof of Theorem~\ref{ratio_limit_theorem}  is given in Section~\ref{ratio_limit_th_proof}. 

\medskip
Next we apply Theorem~\ref{ratio_limit_theorem} to investigate the Martin boundary of the process $(Z(t))$. From now on, instead of the assumptions (A3) and (A4),  we will assume that 

\medskip 
\noindent
\begin{enumerate} 
\item[(A3')] {\em $0\in E$ and for any $u\in E$, 
\[
E + u \subset E \quad \text{and} \quad p(x+u,y+u) ~\geq~ p(x,y), \quad \forall x,y\in E,
\]}
\item[(A4')] {\em for any $u\in E\setminus\{0\}$, there is $\tilde{n}_u \in\N$ such that 
\[
\inf_{x\in E} \P_x(Z(\tilde{n}_u) = x+u) ~>~ 0.
\]}
\end{enumerate} 
The Martin kernel $K(x,y)$ will be defined with the reference point $x_0 = 0$:
\[
K(x,y) = \frac{G(x,y)}{G(0,y)}, \quad x,y\in E. 
\]
As a straightforward consequence of Theorem~\ref{ratio_limit_theorem} we obtain 
\begin{cor}\label{ratio_limit_cor} If a sequence  $(y_n)\in E^\N$  with $\lim_n\|y_n\| = +\infty$ satisfies \eqref{main_result_e3}, then under the hypotheses (A1), (A2), (A3') and (A4'), for any convergent in the Martin compactification $E_M$ subsequence $(y_{n_k})$,
the limite function 
\[
K(z, \eta) ~=~ \lim_{k\to\infty} K(z, y_{n_k}), \quad z,u\in E
\]
satisfies the inequality 
\be\label{main_result_e5} 
K(z + u, \eta) ~\geq~K(z,\eta), \quad \forall z\in E. 
\ee
\end{cor} 
Now we investigate the harmonic  functions $K(\cdot,\eta)$ satisfying \eqref{main_result_e5} by using the method of the ladder height process associated with the Markov chain $(Z(t))$. 
Before formulating our next result we recall the definition and some useful properties of the ladder height process obtained in my previous paper \cite{Ignatiouk-ladder_heights}. 
For this it is convenient to use the following notations~: for  a non-negative function $\varphi : E \to \R_+$, we let
\[
G\varphi(x) ~=~ \sum_{y\in E} G(x,y)\varphi(y), \quad x\in E, 
\]
and 
\[
P\varphi(x) ~=~ \E_x(\varphi(Z(1))) ~=~ \sum_{y\in E} p(x,y) \varphi(y), \quad x\in E.
\]
We denote by $\1$ the identity constant function on $E$ :  $\1(x) = 1$ for all $x\in E$. For a given subset $A\subset E$ we let $\1_A(x)= 1$ if $x\in A$, and $\1_A(x) = 0$ otherwise.  $\Id$ denotes the identity operator : $\Id\varphi = \varphi$ for any function $\varphi:E \to \R$. 
For a given $u\in E$, we define two operators $\varphi \to T_u\varphi$ and $\varphi \to A_u\varphi$ on the set of non-negative functions $\{\varphi : E \to \R_+\}$, by letting 
\[
T_{u}\varphi (x) ~=~ \varphi(x + u), \quad x\in E, 
\]
and 
\be\label{main_result_e6} 
A_{u} \varphi(x) ~=~  \sum_{y\in E} a_{u}(x,y)\varphi(y), \quad x\in E, 
\ee
with  
\[
a_{u}(x,y) = \begin{cases}p(x+ u, y) - p(x,y - u), &\text{if $y \in E + u$,} \\
p(x+ u, y) &\text{otherwise.} 
\end{cases} 
\]
Remark that because of the assumption (A3'),  $a_{u}(x,y) \geq 0$ for all $x,y\in E$, and hence, for  any non-negative function $\varphi : E \to \R_+$, the function $A_{u}\varphi : E \to \R_+\cup\{+\infty\}$ is  well defined. 
The matrix $P_H = \left(p_H(x,y), \, x,y\in E\right)$ is defined by 
\be\label{main_result_e7}
p_H(x,y)  =  GA_{x}\1_{\{y\}}(0) ~=~ \sum_{z\in E} G(0,z)a_x(z,y), \quad x,y\in E,  
\ee
To introduce the ladder height process we use  the following results of the paper \cite{Ignatiouk-ladder_heights} (see Lemma~3.1 of \cite{Ignatiouk-ladder_heights})~: 

\begin{lemma}\label{lemma1-1}  Under the hypotheses (A1) and  (A3'), the matrix $P_H$ is sub-stochastic.
\end{lemma}

\noindent 
As well as for the sub-stochastic transition matrix of the process $(Z(t))$ we introduce an additional  state $\vartheta$ by letting 
\[
p_H(x,\vartheta) = 1 - \sum_{y\in E} p_H(x,y) \quad \text{and} \quad p_H(\vartheta,\vartheta) = 1. 
\]
Without any restriction of generality one can assume that this additional state $\vartheta$ is the same as for the killed random walk $(Z(t))$. 
\begin{defi}
A ladder heights process $(H(n))$  relative to $(Z(t))$  is defined as a Markov chain on $E\cup\{\vartheta\}$ with transition probabilities $p_H(x,y)$, $x,y\in E\cup\{\vartheta\}$. 
\end{defi} 
\noindent 
In a particular case, when $p(x+u,y+u) ~=~ p(x,y)$ for all  $x,y,u\in E$, i.e. if $(Z(t))$ is a copie of a homogeneous random walk on $\Z^d$ killed up on the first exit from the set $E\subset \Z^d$, there is another equivalent definition of the ladder height process $(H(n))$ : for a sequence of random times $(t_n)$ defined  by 
\be\label{eq1-1001} 
t_0 = 0, \quad \text{and} \quad  t_{k+1} = \begin{cases}\inf\{ n > t_k : Z(n) \not\in E + Z(t_k)\}, &\text{if $t_k <+\infty$},\\ +\infty &\text{otherwise,}
\end{cases}
\ee
in distribution 
\be\label{e1-10}
H(k) = \begin{cases} X(t_k) &\text{if $t_k < \infty$}\\
\vartheta, &\text{otherwise}.
\end{cases} 
\ee 
(see Proposition~3.2 of the paper \cite{Ignatiouk-ladder_heights} for more details). 

\noindent 

For $Z(0) = H(0) \in E$ we consider two stopping times $\tau$ and ${\cal T}$ defined as follows :  $\tau = \inf\{t\geq 1 :~Z(t)=\vartheta\}$ is the time of the first exit of the Markov chain $(Z(t))$ from the set $E$, and  ${\mathcal T} = \inf\{n > 0 ~:~  H(n) = \vartheta\}$ is first time when the process $(H(n))$ exits from $E$.
\begin{defi} 
The renewal function $V: E\cup\{\vartheta\} \to \R_+\cup\{+\infty\}$ is then defined by 
\[
V(x) = \begin{cases}\E_x({\mathcal T}) &\text{if $ x\in E$}, \\
0  &\text{otherwise},
\end{cases} 
\] 
where $\E_x(\cdot)$ denotes the  expectation with respect to the probability measure $\P_x$ on the space of trajectories of the processes corresponding to the initial state $H(0) = Z(0) = x$. 
\end{defi} 

Recall that for a  Markov chain $(Z(t))$, a  non-zero positive function $h: E\to \R_+$  is called  super harmonic if 
$P h(x) ~\leq~ h(x)$ for all $x\in E$. A  function $g : E \to \R_+$  is called potential for $(Z(t))$ if for any $x\in E$, $g(x) = G\varphi(x)$ with some non-negative function $\varphi : E \to \R_+$. Such a  function $\varphi : E \to \R_+$ is then uniquely determined by the following relation 
\[
\varphi =  g - P g.
\]
Any potential function is super harmonic, and by the Riesz decomposition theorem, any super harmonic function $f$ is equal to a sum of a harmonic function $h = \lim_n P^n f$ and a potential  function $g = G\varphi$ with $\varphi = (\Id-P) f$, see for instance Woess~\cite{Woess}. 

Remark that because of the assumption (A1), the function  $x\to \E_x(\tau)$ is either finite everywhere on $E$, or infinite also everywhere on $E$.

Theorem~2 of \cite{Ignatiouk-ladder_heights}  proves the following statement :

\begin{theorem}\label{th2} Under the hypotheses (A1) and (A3'),   
\begin{enumerate} 
\item[(i)] The function $V$ is finite with $V(0)=1$.
\item[(ii)] If $\E_\cdot(\tau) = +\infty$, the function $V$ is  harmonic for the Markov chain $(Z(t))$ and for  any $x\in E$,
\[
\lim_{n\to\infty} {\P_x(\tau > n)}/{\P_0(\tau > n)} ~=~ V(x).  
\]
\item[(iii)] If $\E_\cdot(\tau) < +\infty$,  the function $V$ is  potential for the Markov chain $(Z(t))$ and  for  any $x\in E$,  
\[
V(x) ~=~ {\E_x(\tau)}/{\E_e(\tau)}  ~\leq~ \liminf_{n\to\infty} ~{\P_x(\tau > n)}/{\P_e(\tau > n)}.  
\]
\end{enumerate} 
\end{theorem} 
Using this theorem we obtain 

\begin{theorem}\label{main_result_th3} Suppose that a harmonic for $(Z(t))$ function $h: E\to \R_+$ satisfies the inequality 
\be
h(z+u) ~\geq~ h(x), \quad \forall z,u \in E. 
\ee
Then under the hypotheses (A1) and (A3'), the following assertions hold. 
\begin{enumerate}
\item If $\E_\cdot(\tau) = +\infty$, then the function $h$ is proportional to $V$ and for any $x\in E$, 
\[
h(x) ~=~ h(0) V(x) ~=~ h(0) \lim_{n\to\infty} {\P_x(\tau > n)}/{\P_0(\tau > n)}. 
\]
\item If $\E_\cdot(\tau) < +\infty$, then  $h\geq h(0)V$ and the function $h - h(0)V$ is non-trivial. 
\end{enumerate} 
\end{theorem} 
The proof of Theorem~\ref{main_result_th3} is given in Section~\ref{main_result_th3_proof}.

When combined with Corollary~\ref{ratio_limit_cor}, Theorem~\ref{main_result_th3} implies 

\begin{cor}\label{main_result_cor2}  Suppose that  a sequence  $(y_n)\in E^\N$  with $\lim_n\|y_n\| = +\infty$ satisfies  \eqref{main_result_e3} and converges in the Martin compactification to some point $\eta\in\partial_M E$. Suppose moreover that the limit function 
\[
K(z, \eta) ~=~ \lim_{n\to\infty} K(z, y_{n}), \quad z\in E
\]
is harmonic for $(Z(t))$. 
Then under the hypotheses (A1), (A2), (A3') and (A4'), 
\[
K(\cdot, \eta) = 
\tilde{h}_\eta + V ~\geq~ V
\]
where $\tilde{h}_\eta = 0$ if $\E_\cdot(\tau) = +\infty$, and $\tilde{h}_\eta\not= 0$ whenever $\E_\cdot(\tau) < +\infty$.
\end{cor}

\bigskip

Now we apply the above results  to investigate the limit functions $K(\cdot,\eta)$ for a homogeneous random walk in $\Z^d$ killed upon the first exit from a convex cone. 
Consider  a probability measure $\mu$ on the lattice $\Z^d$ and let $(X(t))$ be a homogeneous random walk on  $\Z^d$ with transition probabilities 
\[
\P_x(X(1)=y)  = \mu(y-x), \quad x,y \in \Z^d.
\]
Denote by $\tau$ the first time when the random walk  $(X(t))$ exits from the cone ${\cal C}$~: 
\[
\tau ~=~ \inf\{t\geq 0~: X(t)\not\in{\cal C}\},
\]
and let  $(Z(t))$ be a copie of the random walk $(X(t))$ killed upon the time $\tau$. $(Z(t))$  is then a sub-stochastic random walk on ${E} = {\cal C}\cap\Z^d$ with transition probabilities $p(x,y)=\mu(y-x)$, $x,y\in{E}$. We introduce for $(Z(t))$ an additional absorbing state $\vartheta$ by letting 
\[
p(x,\vartheta) = 1 - \sum_{y\in E} p(x,y),
\]
so that $\tau ~=~ \inf\{t\geq 0~: Z(t) = \vartheta\}$.  As above, $\P_z$ denotes the probability measure on the space of trajectories of $(X(n))$ and $(Z(n))$ corresponding to the initial state $Z(0) = X(0) = x$, $\E_z$ denotes the expectation with respect to the measure $\P_z$, and  $G(x,y)$ and $K(x,y)$ denote respectively the Green function  and the Martin kernel associated with the random walk $(Z(t))$. 

We will assume that 

\begin{enumerate} 
\item[(B0)] the cone ${\cal C}$ is the closure of an open convex cone ${\cal C}^\circ\subset \R^d$ having a vertex at $0$; 
\item[(B1)] the random walk $(Z(n))$ is transient  on $E~=~ {\cal C}\cap\Z^d$ and satisfies the following communication condition~:  there are $\kappa_0 >0$ and a finite set ${\cal E}_0 \subset \Z^d$ such that 
\begin{itemize}
\item[(a)] $\mu(x) > 0$ for all $x\in{\cal E}_0$; 
\item[(b)] for any $x\not= y$, $x,y\in E$ there exists a sequence $x_0, x_1, \ldots , x_n\in E$ with $x_0=x$, $x_n=y$ and $n\leq \kappa_0 |y-x|$ such that $x_j-x_{j-1}\in {\cal E}_0$ for all $j\in\{1,\ldots,n\}$;
\end{itemize} 
\item[(B2)] the step generating function 
\[
R(\alpha) = \sum_{x\in\Z^d} \exp( \alpha\cdot x  ) \mu(x), \alpha\in \R^d,
\]
is finite in a neighborhood of the set $D = \{\alpha\in \R^d~:~ R(\alpha) \leq 1\}$;
\item[(B3)]the mean step of the random walk $(X(n))$ is non-zero: 
\[
m ~=~ \sum_{x\in\Z^d} x \mu(x) \not= 0.
\]
\end{enumerate} 

Under the above assumptions,  the set $D$ is compact and convex, the gradient  $
\nabla R(a)$ exists everywhere on $\R^d$ and does not vanish on the boundary $\partial D \dot= \{\alpha :
R(\alpha) = 1\}$, and the mapping
\be\label{e1def}
 q (\alpha) = {\nabla R(\alpha)}/{|\nabla R(\alpha)|}
\ee
determines a homeomorphism between  $\partial D$ and the unit sphere ${\cal S}^d$ in $\R^d$, (see~\cite{Hennequin}). We denote by $q\to \alpha(q)$ the inverse mapping to $q(\cdot): \partial D\to S$. Then for any $q\in S$, the point $\alpha(q)\in\partial D$ is  the unique point of the set $D$ that achieves the maximum of the function $\alpha\cdot q  $ over $\alpha\in D$. 
We let   
\[
S_+ = S\cap{\cal C} \quad \text{and} \quad  \partial_+ D ~=~\{\alpha\in\partial D~:~ q(\alpha)\in{\cal C}\}.
\]
The mappings  $q\to\alpha(q)$ determines then  a homeomorphism  from $S_+$ to $\partial_+D$. 

For a given $\alpha\in\partial _+D$  we consider a twisted homogeneous random walk $(X_\alpha (t))$ on $\Z^d$, with transition probabilities 
\[
p_\alpha(x,y) = \exp(\alpha \cdot(y-x)  ) \mu(y-x), \quad x,y\in\Z^d,
\]
and a copie $(Z_\alpha (t))$ of $(X_\alpha (t))$ killed upon the time $
\tau_{(\alpha)} = \inf\{t \geq 0 : X_\alpha (t) \not\in{E}\}$. 

Remark that under the hypotheses (B0)-(B3), the killed twisted random walk $(Z_\alpha (t))$ satisfies the hypotheses (A1), (A2), (A3') and (A4'), and hence, the ladder height process $(H_\alpha (n))$  related to the killed twisted random walk $(Z_\alpha (t))$ and the corresponding renewal function  $V_\alpha$ are well defined~: 
\[
V_\alpha(x) ~=~ \E_x\left({\cal T}_\alpha\right), \quad x\in E,
\]
where ${\cal T}_\alpha = \inf\{k > 0~: H_\alpha (k) = \vartheta\}$.  Moreover, in this case, for a sequence of stopping times $(t_n^{\alpha})$ defined  by 
\[
t_0^{\alpha} = 0, \quad \text{and} \quad  t_{k+1}^{\alpha} = \begin{cases} \inf\left\{ t > t_k^{\alpha} : Z_\alpha (t) \not\in E + Z_\alpha (t_k^\alpha)\right\} &\text{ if $t_k^\alpha <\infty$,}\\ 
+\infty &\text{otherwise} 
\end{cases} 
\]
 in distribution, 
\be\label{eq1-60}
H_\alpha (k) = \begin{cases} Z_\alpha (t_k^{\alpha}) &\text{if $t_k^{\alpha} < \infty$}\\
\vartheta, &\text{otherwise}, 
\end{cases} 
\ee 
for all $k\in\N$. Recall moreover that because of Assumption (B1), for any $\alpha\in\partial D$, the twisted killed random walk $(X_+^{(\alpha)}(n))$ is irreducible in ${E}$, and consequently, the function $x\to \E_x(\tau_{\alpha})$ is either finite everywhere in $E$ or infinite also everywhere in $E$. We let  
\[
\partial_+^\infty D ~=~ \{\alpha\in\partial_+D :~ \E_\cdot (\tau_\alpha ) = +\infty\} \quad \text{and} \quad S_+^\infty = \{q\in S_+ :~ \alpha(q)\in \partial_+^\infty D \}, 
\]
and for $q\in S_+^\infty$, we define the function $k_q: E\to \R_+$ by letting 
\[
k_q(x)= \exp(\alpha(q)\cdot x  ) V_{\alpha(q)}(x), \quad x\in{E}. 
\]
Our first result concerning the Martin boundary of the killed random walk is the following theorem. 
\begin{theorem}\label{theorem_4}  Under the hypotheses (B0)-(B3), for any $q\in S_+^\infty$, the following assertions hold~:
\begin{enumerate}
\item for any  $x,y\in{\cal C}\cap\Z^d$, 
\be\label{theorem4_e0} 
k_q(x+y) ~\geq~\exp(\langle\alpha(q),x\rangle) k_q(y), 
\ee
\item $k_q$ is a finite non-zero harmonic  function for $(Z(n))$,
\item  for any sequence of points $(y_n)\in({E})^\N$ with $\lim_n\|y_n\| = \infty$ and $\lim_n y_n/\|y_n\| = q$, 
\be\label{theorem_4_e1} 
\lim_n K(x,y_n) = k_q(x), \quad \forall x\in{E}. 
\ee
\end{enumerate} 
\end{theorem} 
The proof of this theorem is given in Section~\ref{theorem_4_proof}. 

Now we consider a particular case, when the boundary of the cone ${\cal C}$ is $C^1$. To investigate this case we need the following statement.

\begin{theorem}\label{main_result_example_th1} Under the hypotheses (B0)-(B3), if the relative boundary $\partial S_+$ of the set $S_+ = {\cal C}\cap S$ is $C^1$ and $m\in {\cal C}$, then $\E_x(\tau) = \infty$ for all $x\in E$. 
\end{theorem} 
The proof of this theorem is given in Section~\ref{th1_example_proof}. 

When combined with our previous results, this theorem provides the following statement.

\begin{theorem}\label{main_result_example_th2} Suppose that the relative boundary $\partial S_+$ of the set $S_+ = {\cal C}\cap S$ is $C^1$. Then under the hypotheses (B0)-(B3), 
\begin{enumerate} 
\item[i)] $S_+^\infty = S_+$.
\item[ii)]  For any $q\in S_+$, the  $k_q$ is a  finite  non-zero harmonic function for $(Z(t))$. 
\item[iii)] for any $q\in S_+$ and any sequence of points $(y_n)\in({E})^\N$ with $\lim_n\|y_n\| = \infty$ and $\lim_n y_n/\|y_n\| = q$,  \eqref{theorem_4_e1} holds. 
\end{enumerate} 
\end{theorem}
The proof of this theorem is given in Section~\ref{th2_example_proof}.

\section{Proof of Theorem~\ref{ratio_limit_theorem}}\label{ratio_limit_th_proof}  

Remark that for any $z,z',y\in E$, 
\[
G(z,y) ~\geq~ \P_z(Z(t)=z' \; \text{for some $t\geq 0$}) G(z',y) 
\]
where because of Assumption (A1), 
\[
\P_z(Z(t)=z' \; \text{for some $t\geq 0$}) > 0.
\]
Hence, the inequality \eqref{main_result_e3} implies that
\be\label{main_result_e3'}
\liminf_{n\to\infty} \frac{1}{\|y_n\|}\log G(z,y_n) \geq 0, \quad \forall z\in E. 
\ee
Because of \eqref{main_result_e3'},  for any $z\in E$, the terms of the order $c\exp(-\delta \|y_n\|)$ give an
asymptotically negligible contribution to $G(z,y_n)$. The following lemma provides the first negligible part of $G(z,y_n)$. 

\begin{lemma}\label{rl_lemma1} For any $0\leq \delta < \delta_0$ there are two constants $\kappa > 0 $ and $C > 0$ such that 
for any $z\in E$ and $n\in\N$, 
\be\label{ratio_limit_eq1p} 
 \sum_{0\leq t \leq \kappa \|y_n\|} \P_{z}\bigl(Z(t) = y_n\bigr) ~\leq~ C \exp\left(-  \frac{\delta}{2} \|y_n\| + \delta \|z\|\right)
\ee
\end{lemma} 
\begin{proof} Because of the assumption (A2),  for any $0 < \delta \leq \delta_0$, 
\[
C_\delta ~\dot=~\sup_{\alpha\in\R^d~: \|\alpha\| \leq \delta} ~\sup_{x\in E} \E_x\Bigl(\exp\bigl(\alpha\cdot  (Z(1) -x) \bigr)\Bigr) < \infty.
\]
Using Chebychev's inequality and Markov property  from this it follows that 
for any $z\in E$ and  $\alpha\in\R^d$ with $\|\alpha\| \leq \delta \leq 
\delta_0 $, the following relation holds 
\begin{align*}
\P_{z}\left(\alpha\cdot  Z(t)   \geq  \delta \|y_n\|\right) &\leq~ \exp(-\delta \|y_n\|) 
~\E_{z}\left( \exp(\alpha\cdot  Z(t)  ) \right) \\ &\leq C_\delta^t   \exp(-\delta \|y_n\|  + \alpha\cdot z  ) \leq C_\delta^t \exp\left(- \delta \|y_n\| + \delta \|z\|\right), \quad \forall t\in\N.
\end{align*}
Using this inequality with $\alpha = \delta y_n/\|y_n\|$ we obtain   
\[
\P_{z}\bigl(Z(t) = y_n\bigr) ~\leq~  C_\delta^t \exp\left(- \delta \|y_n\| + \delta \|z\|\right)
\]
and consequently, for $\kappa = \delta/(2 \ln C_\delta)$, 
\begin{align*}
\sum_{0\leq t \leq \kappa \|y_n\|} \P_{z}\bigl(Z(t) = y_n\bigr) 
&~\leq~ \exp\bigl(-\delta \|y_n\| + \delta \|z\|\bigl) \sum_{0\leq t \leq \kappa  \|y_n\| }
C_\delta^t   \\
&~\leq~ \exp\left(-  \frac{\delta}{2} \|y_n\| + \delta \|z\|\right)/(C_\delta-1) , 
\end{align*}
for any $z\in E$ and $n\in\N$.  
\end{proof}
\medskip 

Next, we adapte the method of Bernoulli part decomposition and we begin our analysis with a particular case when 
\[
\eps ~\dot= \inf_{x\in E} \min\{p(x,x), \, p(x,x+u)\} ~>~0. 
\]
Consider a time-homogeneous discret time Markov chain $(W(t), \xi(t),\zeta(t))$ on the set of states $(E\cup\{\vartheta\})\times\{0,1\}\times\{0,1\}$   such that 
 \begin{enumerate}
\item[(i)] $(\xi(k), k \geq 0)$ and $(\zeta(k), k\geq 0)$ are mutually independent sequences of 
independent  Bernoulli random variables with means $\E(\xi(k))= 2\eps$ and $\E(\zeta(k))=1/2$.
\item[(ii)]   the state $\vartheta$ for the process $(W(t))$ is absorbing: if $W(t) = \vartheta$ for some $t\in\N$, then almost surely $W(t')=\vartheta$ for all $t'\geq t$. 
 \item[(iii)] if $\xi(t)=1$ and $W(t)\in E$ then 
 \[
 W(t+1) ~=~\begin{cases} W(t) &\text{whenever $\zeta(t)=0,$}\\
 W(t) + u &\text{whenever $\zeta(t)=1$,} 
 \end{cases} 
 \]
 \item[(iv)] if $\xi(t) = 0$ and $W(t)\in E$, then for any $x\in E$  and $\zeta\in\{0,1\}$, 
 \begin{align*} 
\P(W(t+1) = y\, \vert\, &W(t) = x, \xi(t) = 0, \; \zeta(t) =\zeta  \bigr) \\ &=~ \begin{cases}  
p(x,y)/(1-2\eps) &\text{if $y \not\in\{x, x + u\}$, $x\in E$,}\\
(p(x,y) - \eps)/(1-2\eps) &\text{if  $y \in\{x, x + u\}$, $x\in E$.} 
\end{cases} 
\end{align*} 
\end{enumerate} 
Then  in distribution 
\[
Z(t) = W(t), \quad \forall t\in\N,
\]
and consequently,
\[
G(z,y_n)  ~=~ \sum_{t \geq 0} \P_{z}\left(W(t) = y_n\right). 
\]
For $t \in\N$ we let 
\[
N_t = \sum_{k=0}^{t-1} \xi(k) \quad \text{and} \quad L_t = \sum_{k=0}^{t-1} \xi(k)\zeta(k),
\]
and for  given $0 < \sigma < 1/2$ and $z\in E$ we define 
\[
\Xi_\sigma(z,y_n) ~=~\sum_{t >
  \kappa \|y_n\|} \P_{z}\left(W(t) = y_n, \; \left|L_t -
  N_t/2\right| \leq \sigma N_t, \; N_t \geq \eps t/2 \; \right), 
\]
so that 
\begin{align*}
G(z,y_n) - \Xi_\sigma(z,y_n) &=~ \sum_{t \leq 
  \kappa \|y_n\|} \P_{z}\left(Z(t) = y_n\right) ~+~ \sum_{t >
  \kappa \|y_n\|} \P_{z}\left(W(t) = y_n, \; N_t < \eps t/2 \; \right) \\ 
  &~+ \sum_{t >
  \kappa \|y_n\|} \P_{z}\left(W(t) = y_n, \; N_t \geq \eps t/2, \; \left|L_t -
  N_t/2\right| > \sigma N_t \; \right) 
\end{align*} 
The following lemma proves that for any $\kappa > 0$ and $z\in E$, the part  
\[
\sum_{t >
  \kappa \|y_n\|} \P_{z}\left(W(t) = y_n, N_t < \eps t/2\right) + \sum_{t >
  \kappa \|y_n\|} \P_{z}\left(W(t) = y_n, N_t \geq \eps t/2, \; \left|L_t -
  N_t/2\right| > \sigma N_t \right) 
  \]
 of $G(z,y_n)$ is also negligible. 
\begin{lemma}\label{rl_lemma2}  Suppose that 
\[
\eps ~\dot= \inf_{x\in E} \min\{p(x,x), \, p(x,x+u)\} ~>~0.
\]
Then for any $0 < \sigma < 1/2$ and $\kappa > 0$, there are two constants $\theta > 0$ and $C > 0$ such that for any  $z\in E$, 
\begin{align}
 \sum_{t >
  \kappa \|y_n\|} &\P_{z}\left(W(t) = y_n, \; N_t < \eps t/2 \; \right) \nonumber\\ & + \sum_{t >
  \kappa \|y_n\|} \P_{z}\left(W(t) = y_n, \; N_t \geq \eps t/2, \; \left|L_t -
  N_t/2\right| > \sigma N_t \; \right) ~\leq~ C\exp(-\theta  \|y_n\|) \label{ratio_limit_eq3pp}
\end{align} 
\end{lemma} 
\begin{proof} Remark that $N_t$ is a Binomial random variable with mean $\eps n$ and variance $\eps(1-\eps)n$ and 
by Chebychev's inequality, 
\[
\P(N_t < \eps t/2) ~\leq~ \inf_{\eta < 0}  ~e^{- \eta \eps t/2} \E\left(e^{  \eta
  N_t}\right) ~=~ \exp\left( - t \theta_1\right)
\]
where 
\[
\theta_1 ~\dot=~ \sup_{\eta < 0} \bigl( \eta \eps /2 - \log (\eps e^{\eta} +
  1-\eps)\bigr)  ~>~ 0
  \] 
because the function $f_1(\eta) = \eta \eps /2 - \log (\eps e^{\eta} +
  1-\eps)$ is concave, $f'_1(0) = -\eps/2 < 0$ and $f_1(0)=0$.  From this it follows that 
\begin{align}
\sum_{t >
  \kappa \|y_n\|}  \P_{z}(W(t) = y_n, \; N_t < \eps t/2 ) 
&~\leq~ \sum_{t >
  \kappa \|y_n\|}  ~\P\left( N_t < \eps t/2\right) \nonumber \\ &~\leq~ \exp\left( - \kappa
  \theta_1 \|y_n\|\right)/(1-\exp( - \theta_1))  \label{ratio_limit_eq3p} 
\end{align}
Remark furthermore that the conditional distribution of the random variable $L_t$ 
given that $N_t=N$ is binomial with  mean $N/2$ and variance $N/4$. Hence, for $0 < \sigma < 1/2$, 
\begin{align}
\P\left(\left|L_t - {N_t}/{2 }\right| > \sigma N_t \, \Bigl\vert\, N_t=N\right) &=~ 2 \P\left( \sum_{s=1}^{N}
  \zeta(s) >  \frac{N}{2 } + \sigma N \right)\nonumber\\
&\leq~ 2 ~\inf_{\eta > 0}  ~e^{- \eta (1/2 + \sigma )N} \E\left(\exp\left( \eta \sum_{s=1}^{N}
  \zeta(s)\right)\right) \nonumber\\
&\leq 2 ~\exp\left(- N \theta_2\right)\label{ratio_limit_eq5p} 
\end{align}
where 
\[
\theta_2 ~\dot=~ \sup_{\eta > 0}  \left(\eta (1/2 + \sigma ) - \log \bigl((e^\eta +
1)/2\bigr)\right) ~>~ 0
\]
because the function $f_2(\eta) = \eta (1/2 + \sigma ) - \log \bigl((e^\eta +
1)/2\bigr)$ is concave, $f_2(0)=0$ and $f_2'(0) = \sigma > 0$. From \eqref{ratio_limit_eq5p} it follows that 
\[
\P\left(\left| L_t - {N_t}/{2 }\right| > \sigma N_t, \, N_t \geq \eps t/2 \right) ~\leq~ \E\left( \exp\left(- N_t \theta_2\right); \, N_t \geq \eps t/2  \right) ~\leq~ \exp\left(- \eps \theta_2 t/2\right)
\]
and consequently, 
\begin{align*}
\sum_{t >
  \kappa |z_n|} &\P_{z}\left(W(t) = z_n, \; \left|L_{t} - N_t/2\right| > \sigma N_t, \; N_t \geq \eps t/2 \; \right) \\ 
&\leq \sum_{t >
  \kappa |z_n|}  ~\P\bigl( \left| L_t -
  N_t/2\right| > \sigma N_t, \, N_t \geq \eps t/2 \bigr) ~\leq~ 2 \sum_{t >
  \kappa |z_n|} \exp\left(- \eps \theta_2 t/2\right) \\
&\leq ~\exp\left(-
  \eps \theta_2 \kappa |z_n|/2\right)/(1-\exp(-
  \eps \theta_2/2))
\end{align*}
When combined with \eqref{ratio_limit_eq3p}, the last relation completes the proof of 
\eqref{ratio_limit_eq3pp}.
\end{proof} 

Now we  compare the quantities $\Xi_\sigma(z,y_n)$ and $G(z+u,y_n)$. For this it is convenient to  introduce  two sequences of random sets 
\[
A_t = \{ k\in\{0,\ldots,t-1\}~: \xi(k) = 1\} \; \text{and} \; B_t = \{ k\in\{0,\ldots,t-1\}~: \xi(k) \zeta(k)  = 1\}, \quad t\in\N,
\]
so that 
\[
N_t ~=~\sum_{k=0}^{t-1} \xi(k) = Card(A_t) \quad \text{and} \quad L_t = \sum_{k=0}^{t-1} \xi(k) \zeta(k) = Card(B_t). 
\]
\begin{lemma}\label{ratio_limit_lem3} For any subsets $B\subset\ A\subset\{0,\ldots,t-1\}$, $k\in B$ and $z\in E$, 
\be\label{ratio_limit_eq9}
\P_z(W(t) = y_n,  A_t = A,  B_t = B) ~\leq~  \P_{z+u}(W(t) = y_n,  A_t = A,  B_t = B\setminus\{k\}) 
\ee
\end{lemma} 
\begin{proof} For  given $z,y\in E$ and $B\subset\ A\subset\{1,\ldots,t\}$ , denote by $\Gamma_{A,B}(z,y)$ the set of all sequences $z_0,\ldots,z_t\in E$ with $z_0=z$ and $z_t=y$ such that 
\[
z_{i+1} = \begin{cases} z_i &\text{for all $i\in A\setminus B$,}\\
z_i + u &\text{for all $i\in B$}
\end{cases} 
\]
Then according to the definition of the Random process $(W(t),\xi(t),\zeta(t))$ and the random sets $A_t$ and $B_t$,  the left hand side of \eqref{ratio_limit_eq9} is equal to 
\[
\sum_{(z_0,\ldots, z_t)\in\Gamma_{A,B}(z,y_n)} \eps^{Card(A)}\times \prod_{i\in\{0,\ldots,t-1\}\setminus A} \bigl(p(z_i,z_{i+1}) - \eps\1_{\{0, u\}}(z_{i+1}-z_i)\bigr), 
\]
where for $z\in\Z^d$, 
\[
\1_{\{0, u\}}(z) ~=~ \begin{cases} 1 &\text{if $z\in\{0,u\}$},\\ 
0 &\text{otherwise} 
\end{cases} 
\]
and the right hand side of \eqref{ratio_limit_eq9} is equal to 
\[
\sum_{(\tilde{z}_0,\ldots, \tilde{z}_t)\in\Gamma_{A,B\setminus\{k\}}(z+u,y_n)} \eps^{Card(A)}\times \prod_{i\in\{0,\ldots,t-1\}\setminus A} \bigl(p(\tilde{z}_i,\tilde{z}_{i+1}) - \eps\1_{\{0, u\}}(\tilde{z}_{i+1}-\tilde{z}_i)\bigr). 
\]
Define a mapping $(z_0,\ldots, z_t) \to (\tilde{z}_0,\ldots, \tilde{z}_t)$ from $\Gamma_{A,B}(z,y_n)$ to $\Gamma_{A,B\setminus\{k\}}(z+u,y_n)$ by letting 
\[
\tilde{z}_i ~=~\begin{cases} z_i + u&\text{if $i \leq k$}\\
z_i &\text{if $i > k$} 
\end{cases} 
\]
Then for any $(z_0,\ldots, z_t)\in\Gamma_{A,B}(z,y_n)$, 
\[
z_{i+1} - z_i ~=~ \tilde{z}_{i+1} - \tilde{z}_i \quad \forall i\not= k, 
\]
and because of Assumption (B3), 
\[
p(\tilde{z}_i,\tilde{z}_{i+1}) = \begin{cases} p(z_i,z_{i+1})  &\text{for $i > k$,}\\
 p(z_i + u ,z_{i+1}+ u) ~\geq~ p(z_i,z_{i+1})  &\text{for $i < k$}. 
 \end{cases} 
\]
Since $k\in A$, from this it follows  that for any $(z_0,\ldots, z_t)\in\Gamma_{A,B}(z,y_n)$,
\[
\prod_{i\in\{0,\ldots,t-1\}\setminus A} \!\!\bigl(p(z_i,z_{i+1}) - \eps\1_{\{0, u\}}(z_{i+1}-z_i)\bigr) \leq \prod_{i\in\{0,\ldots,t-1\}\setminus A} \!\!\bigl(p(\tilde{z}_i,\tilde{z}_{i+1}) - \eps\1_{\{0, u\}}(\tilde{z}_{i+1}-\tilde{z}_i)\bigr).
\]
The mapping $(z_0,\ldots, z_t) \to (\tilde{z}_0,\ldots, \tilde{z}_t)$ from $\Gamma_{A,B}(z,y_n)$ to $\Gamma_{A,B\setminus\{k\}}(z+u,y_n)$ being injective, the last inequality completes the proof of \eqref{ratio_limit_eq9}. 
\end{proof} 

As a consequence of Lemma~\ref{ratio_limit_lem3} we obtain. 
\begin{lemma}\label{ratio_limit_lem4} For any $z\in E$ and $N,L\in\N$ such that $N \geq L\geq 1$, the following inequality holds 
\begin{multline}\label{ratio_limit_eq10}
\P_z(W(t) = y_n,  N_t = N,  L_t = L)  \\\leq \frac{N-L+1}{L}~\P_{z+u}(W(t) = y_n,  N_t = N, L_t = L-1)
\end{multline} 
\end{lemma} 
\begin{proof} Let $({\cal F}_t)$ be the natural filtration of the process $(W(t),\xi(t),\zeta(t))$. 
Consider random variables $U_t$ and $V_t$ such that any $t\in\N$,  the conditional distribution of $U_t$ given ${\cal F}_t$  is uniform on the set $A_t\setminus B_t$~:
\[
\P\left( U_t = k \; \Bigl\vert \; {\cal F}_t\right) = \begin{cases} 1/(Card(A_t\setminus B_t))  &\text{if $k \in A_t\setminus B_t$}\\ 
\, \\ 
0 &\text{otherwise},
\end{cases} 
\]
and the conditional distribution of $V_t$ given ${\cal F}_t$ is uniform on the set $B_t$~:
\[
\P\left( V_t = k \; \Bigl\vert {\cal F}_t\right) = \begin{cases} 1/Card(B_t) &\text{if $k\in B_t$}\\ 
\, \\ 
0 &\text{otherwise},
\end{cases} 
\]
Then according to the definition of the random variables $N_t$, $L_t$, $U_t$ and $V_t$, 
\begin{multline}\label{ratio_limit_eq11}
(N-L+1) \, \P_{z+u}(W(t) = y_n,  N_t = N, L_t = L-1, U_t=k)  \\=~   \sum_{\substack{A :~A\subset\{0,\ldots,t-1\} \\Card(A) = N}} \;~\sum_{\substack{B :~B\subset A , k\in A\setminus B \\ Card(B) = L-1}}\P_z(W(t) = y_n,  A_t = A,  B_t = B) 
\end{multline} 
where the summation  is taken over all subsets $B\subset A \subset\{0,\ldots,t-1\}$ with $Card(A) = N$ and $Card(B) = L-1$ and such that $k\in A\setminus B$,  
and similarly 
\begin{multline*}
 L \, \P_z(W(t) = y_n,  N_t = N,  L_t = L, V_t=k)   \\=~   \sum_{\substack{A :~ A\subset\{0,\ldots,t-1\} \\Card(A) = N}} \;~\sum_{\substack{B : ßB\subset A , k\in B \\ Card(B) = L}}\P_z(W(t) = y_n,  A_t = A,  B_t = B) 
\end{multline*} 
where the summation  is taken over all subsets $B\subset A \subset\{0,\ldots,t-1\}$ with $Card(A) = N$ and $Card(B) = L$ and such that $k\in B$. When combined with Lemma~\ref{ratio_limit_lem3} the last relation proves that 
\begin{multline*}
L \, \P_x(W(t) = y_n,  N_t = N,  L_t = L, V_t=k)  \\ \leq   \sum_{\substack{A :~ A\subset\{1,\ldots,t\} \\Card(A) = N}} \; ~\sum_{\substack{B :~ B\subset A , k\in B \\ Card(B) = L}}\P_{z+u}(W(t) = y_n,  A_t = A,  B_t = B\setminus\{k\}) 
\end{multline*} 
Since the right hand side of the last relation is identical to the right hand side of \eqref{ratio_limit_eq11}, we conclude that for any $1 \leq k\leq t$,
\begin{multline*}
\P_z(W(t) = y_n,  N_t = N,  L_t = L, V_t=k)  \\ \leq \frac{N-L+1}{L}~\P_{z+u}(W(t) = y_n,  N_t = N, L_t = L-1, U_t=k)
\end{multline*} 
Taking finally the summation over $k\in\{1,\ldots,t\}$ at the right hand side and the left hand side of the last relation,  one gets \eqref{ratio_limit_eq10}. 
\end{proof} 

Now we are ready to get 
\begin{lemma}\label{ratio_limit_lem5} Suppose that 
\[
\eps ~\dot= \inf_{x\in E} \min\{p(x,x), \, p(x,x+u)\} ~>~0.
\]
Then for any $z\in E$ and $0< \sigma < 1/2$, 
\be\label{ratio_limit_eq12} 
\Xi_\sigma(z,y_n) ~\leq~ \frac{1/2 + \sigma +2/(\eps \kappa \|y_n\|)}{1/2 - \sigma} ~G(z+u,y_n).
\ee
\end{lemma} 
\begin{proof}
Indeed, for any  $0 < \sigma < 1/2$, according to the definition of the quantity $\Xi_\sigma(z,y_n)$ and using Lemma~\ref{ratio_limit_lem4}, one gets 
\begin{align*}
&\Xi_\sigma(z,y_n) = \sum_{t >
  \kappa \|y_n\|} \P_{z}\left(W(t) = y_n, \; \left|L_t -
  N_t/2\right| \leq \sigma N_t, \; N_t \geq \eps t/2 \; \right)\\
  &\quad = \sum_{t >
  \kappa \|y_n\|} ~\sum_{N \geq \eps t/2} ~\sum_{|L - N/2|\leq \sigma N} \P_{z}\left(W(t) = y_n, \, L_t = L, \, N_t = N \right) \\
  &\quad \leq \sum_{t >
  \kappa \|y_n\|} ~\sum_{N \geq \eps t/2} ~\sum_{|L - N/2|\leq \sigma N} \frac{N-L+1}{L} \P_{z+u}\left(W(t) = y_n, \, L_t = L-1, \, N_t = N \right) \\
  &\quad \leq \sum_{t >
  \kappa \|y_n\|} ~\sum_{N \geq \eps t/2} ~\sum_{|L - N/2|\leq \sigma N} \frac{1/2 + \sigma +1/N}{1/2 - \sigma} ~\P_{z+u}\left(W(t) = y_n, \, L_t = L-1, \, N_t = N \right) \\
  &\quad \leq \sum_{t >
  \kappa \|y_n\|} \frac{1/2 + \sigma +2/(\eps t)}{1/2 - \sigma} ~\P_{z+u}\left(W(t) = y_n \right) ~\leq~ \frac{1/2 + \sigma +2/(\eps \kappa \|y_n\|)}{1/2 - \sigma} ~G(z+u,y_n) 
\end{align*} 
\end{proof} 
When combined together, Lemma~\ref{rl_lemma1}, Lemma~\ref{rl_lemma2} and Lemma~\ref{ratio_limit_lem5}  imply the following statement.

\begin{cor}\label{rl_cor} Suppose that 
\[
\eps ~\dot= \inf_{x\in E} \min\{p(x,x), \, p(x,x+u)\} ~>~0.
\]
Then for any $0 < \delta <\delta_0$ and $0 < \sigma < 1/2$, there are $\theta > 0$ and $C >0$ such that for any $z\in E$, 
\be\label{rl_e12} 
G(z,y_n)  ~\leq~ \frac{1 + 2\sigma
  + \theta/\|y_n\|}{1 - 2\sigma }~ G(z + u,y_n) + C \exp(- \theta \|y_n\| + \delta \|z\|). 
\ee
\end{cor} 

\noindent 
Remark that this statement proves \eqref{main_result_e4}  in the case when 
\[
\eps ~\dot= \inf_{z\in E} \min\{p(x,x), \, p(x,x+u)\} ~>~0.
\]
Indeed, in this case, from  \eqref{rl_e12} and \eqref{main_result_e3} it follows that for any $0 < \sigma < 1/2$ 
\[
\liminf_n ~\frac{G(z+u,y_n)}{G(z,y_n)} ~\geq~ \frac{1/2 - \sigma}{1/2 + \sigma}.
\]
Since the left hand side of the last inequality does not depend of $\sigma > 0$, letting at the right hand side $\sigma\to 0$ one gets \eqref{main_result_e4}.

To complete the proof of Theorem~\ref{ratio_limit_theorem}, we need moreover the following lemma. 

\begin{lemma}\label{ratio_limit_lem0}
For any $\varepsilon > 0$, the Green function $\tilde{G}(x,y)$ of the Markov chain $(\tilde{Z}(t))$ with modified transition probabilities 
\[
\tilde{p}(x,y) = \begin{cases} (1-\epsilon)p(x,y) &\text{if $y\not= x$}\\
\epsilon &\text{if $x=y$}
\end{cases}, \quad x,y \in E,
\]
is related to the Green function $G(x,y)$ of the original Markov chain $(Z(t))$ in the following way:
\[
\tilde{G}(x,y) ~=~ (1-\eps)^{-1} G(x,y), \quad \forall x,y\in E. 
\]
\end{lemma}
\begin{proof} Indeed, for $0<\lambda \leq1$, consider the matrices $G_\lambda=(G_\lambda(x,y), \; x,y\in E)$ and $\tilde{G}_\lambda = (\tilde{G}_\lambda(x,y), \; x,y\in E)$ with 
\[
G_\lambda(x,y) = \sum_{n=0}^\infty \lambda^n p^n(x,y) \quad \quad 
\text{and} \quad \quad 
\tilde{G}_\lambda(x,y) = \sum_{n=0}^\infty \lambda^n \tilde{p}^n(x,y),
\]
and let $P=(p(x,y), \; x,y\in E)$ and $\tilde{P} = (\tilde{p}(x,y), \; x,y\in E)$ denote respectively the transition matrices of $(Z(t))$ and $(\tilde{Z}(t))$. 
Then for any $0<\lambda < 1$, the series 
\[
G_\lambda = \sum_{n=0}^\infty \lambda^n P^n  
\quad \text{and} \quad 
\tilde{G}_\lambda = \sum_{n=0}^\infty \lambda^n \tilde{P}^n 
\]
converge with respect to the norm of bounded linear operators on the space of bounded functions $f: E \to \R$ endowed by the norm $\|f\|_\infty = \sup_{x\in E}|f(x)|$ respectively to 
\[
(Id - \lambda P)^{-1} \quad \text{and} \quad (Id - \lambda \tilde{P})^{-1}
\]
and 
\begin{align*} 
\tilde{G}_\lambda &=~ (Id - \lambda\tilde{P})^{-1} ~=~ ((1-\eps\lambda) Id - (1-\eps)\lambda P)^{-1}  \\ &=~ (1-\eps\lambda)^{-1} \left( Id - \frac{(1-\eps)\lambda}{(1-\eps\lambda)} P\right)^{-1}  ~=~  (1-\eps\lambda)^{-1} G_{\lambda_\eps} 
\end{align*} 
where $\lambda_\eps = \frac{(1-\eps)\lambda}{(1-\eps\lambda)} \to 1$ as $\lambda\to 1$. Hence, letting $\lambda\to 1$, one gets 
\[
\tilde{G}(x,y) = \lim_{\lambda\to 1} \tilde{G}_\lambda(x,y) =  \lim_{\lambda\to 1}  (1-\eps\lambda)^{-1} G_{\lambda_\eps}(x,y) = (1-\eps)^{-1} G(x,y), \quad \forall x,y\in E. 
\]
\end{proof}

This lemma shows that to get \eqref{main_result_e4}, without any restriction of generality, one can assume that 
\be\label{rl_e5}
\inf_{x\in E} p(x,x) ~>~0. 
\ee
Then because of the Assumption (A3), 
\[
\inf_{x\in E}  \min\{ p^{(\tilde{n})}(x,x), p^{(\tilde{n})}(x,x +u) > 0 
\]
In the case when  $\tilde{n} = 1$, Theorem~\ref{ratio_limit_theorem} is already proved. Suppose now that $\hat{n} > 1$. Then Corollary~\ref{rl_cor} applied for the Green's function
$\tilde{G}(z,z_n)$ of the embedded Markov chain
$Z(\hat{n}t)$ proves that  for any $0 < \delta < \delta_0$ and $0 < \sigma < 1/2$ there are  $C>0$ and $\theta>0$  such that    
\[
\tilde{G}(z,y_n)  ~\leq~ \frac{1 + 2\sigma
  +\theta/ \|y_n\|)}{1 - 2\sigma }~\tilde{G}(z + u,y_n) + C \exp(- \theta \|y_n\| + \delta \|z\|) 
\] 
for all $z\in E$. Since 
\[
G(z,y_n) ~=~ \sum_{t=0}^{\hat{n}-1} \sum_{z'\in E} p^{(t)}(z,z') 
\tilde{G}(z',y_n) 
\]
and 
\begin{align*}
G(z+ u,y_n) &= \sum_{t=0}^{\hat{n}-1} \sum_{\tilde{z}\in E} p^{(t)}(z+u,\tilde{z}) \tilde{G}(\tilde{z},y_n) \\ &\geq~ \sum_{t=0}^{\hat{n}-1} \sum_{\tilde{z}\in E+u} \!\!p^{(t)}(z+u,\tilde{z}) \tilde{G}(\tilde{z},y_n)  = \sum_{t=0}^{\hat{n}-1} \sum_{\tilde{z}\in E} p^{(t)}\!(z+u,\tilde{z}+u) \tilde{G}(\tilde{z}+u,y_n)\\
&\geq ~\sum_{t=0}^{\hat{n}-1} \sum_{\tilde{z}\in E} p^{(t)}(z,\tilde{z}) \tilde{G}(\tilde{z}+u,y_n) 
\end{align*} 
from this it follows that 
\begin{multline}\label{e8-8}
G(z,y_n) \leq \frac{1 + 2\sigma  + \theta/\|y_n\|)}{1 - 2\sigma }
G(z+u,y_n)   + C \sum_{t=0}^{\hat{n}-1} \sum_{\tilde{z}\in E} p^{(t)}(z,\tilde{z}) \exp\left(\delta |\tilde{z}| - \theta \|y_n\|\right).
\end{multline}
Since under the hypotheses (A3), for any $0 < \delta \leq \delta_0$,   the series 
\begin{align*}
\sum_{z'} p^{(t)}(z,\tilde{z}) \exp\left(\delta |\tilde{z}| \right)
\end{align*}
converge, using \eqref{main_result_e3} we conclude that 
\[
 \lim_{n\to\infty} \frac{1}{G(z,y_n)} \sum_{t=0}^{\hat{n}-1} \sum_{\tilde{z}\in E} p^{(t)}(z,\tilde{z}) \exp\left(\delta |\tilde{z}| - \theta \|y_n\|\right) ~=~0,
\]
and consequently, by~\eqref{e8-8}, for any $0 < \sigma < 1/2$, 
\[
\liminf_{n\to\infty} \frac{G(z+ u,y_n)}{G(z,y_n)} ~\geq~ \frac{1 - 2\sigma }{1 + 2\sigma}. 
\]
Letting finally at the last inequality $\sigma\to 0$ we obtain  \eqref{main_result_e4}.
Theorem~\ref{ratio_limit_theorem} is  therefore proved.

\section{Proof of Theorem~\ref{main_result_th3}.} \label{main_result_th3_proof} 

Let a function $h: E\to\R_+$ be harmonic for $(Z(t))$ and let 
\be\label{ren_func_e1} 
h(x+u) ~\geq~  h(x), \quad \forall x,u\in {E}.
\ee 
We extend this function on $E\cup\{\vartheta\}$ by letting $h(\vartheta)=0$. 
Then because of the assumption (A3), for any $u\in{E}$, the function $T_uh:\Z^d\cup\{\vartheta\}\to \R_+$ defined by 
\[
T_uh(x) = \begin{cases} h(x+u), & \text{for $x\in{E}$}\\
0 &\text{otherwise} 
\end{cases} 
\]
is super-harmonic for $(Z(t))$. By Riesz decomposition theorem from this it follows that 
\[
T_uh = f + G\varphi 
\]
where the function 
\[
f(x) = \lim_{t\to\infty} \E_x(T_uh(Z(t))), \quad x\in\Z^d,
\]
is harmonic for $(Z(t))$ and the function 
\[
G\varphi(x) ~=~ \sum_{t=0}^\infty \E_x(\phi(Z(t))), \quad x\in\Z^d,
\]
is potential for $(Z(t))$ with 
\[
\phi(x) ~=~ T_uh(x) - \E_x(T_uh(Z(1))), \quad x\in{E}. 
\]
Remark  that for any $x\in{E}$, by \eqref{ren_func_e1}, 
\[
f(x) = \lim_{t\to\infty} \E_x(T_uh(Z(t))) ~\geq~ \lim_{t\to\infty} \E_x(h(Z(t))) ~=~ h(x),
\]
and according to the definition of the quantities $a_u(x,y)$, for any $x\in E$, 
\begin{align*} 
\phi(x)~=~T_uh(x) - \E_x(T_uh(Z(1))) &=~ h(x+u)  - \sum_{y\in{E}} p(x,y) h(y+u) \\ &=~  \sum_{y\in{E}} p(x+u,y) h(y) - \sum_{y\in{E}} p(x+u,y+u) h(y+u) \\
&= \sum_{y\in E} a_u(x,y) h(y)\\ &=~ A_uh(x), 
\end{align*} 
from which it follows that for any $x,u\in E$, 
\[
h(x+u) = T_uh(x) ~\geq~ h(x) + GA_uh(x).
\]
Letting in the last relation $x=0$ and using the definition of the ladder height process $(H(n))$,  one gets 
\be\label{ren_func_e2} 
h(u) ~\geq~  h(0) + GA_uh(0) ~=~ h(0) + \E_{u}\left(h(H(1))\right), \quad \forall u\in {E}, 
\ee 
The function $h$ is therefore super-harmonic for the Markov chain $(H(n))$. By the Riesz decomposition theorem, from this it follows that $h=\tilde{h} + g$ where the function 
\[
\tilde{h}(x) = \lim_n \E_{u}\left(h(H(n))\right)
\]
is either zero or positive and harmonic for $(H(n))$ and the function $g$ is potential for $(H(n))$~: 
\[
g(x) ~=~ \sum_{n=0}^\infty \E_x\bigl(\varphi(H(n))\bigr), \quad x\in{E},
\]
with 
\[
\varphi(x) = \E_{x}\left(h(H(1))\right) - h(x) ~\geq~h(0), \quad x\in{E}. 
\]
Hence, for any $x\in{E}$, 
\[
h(x)  ~\geq~ g(x) ~\geq~ h(0) \sum_{n=0}^\infty \P_x\bigl(T > n\bigr) ~=~h(0)V(x).
\]
If $\E_\cdot(\tau) = + \infty$, then by Theorem~\ref{th2}, the function $V$ is harmonic for $(Z(t))$. Using the above inequality we conclude therefore that the function $f= h - h(0)V$ is either zero or non-negative and harmonic for $(Z(t))$ with $f(0)=0$. Since under the hypotheses (A1), the killed random walk $(Z(t))$ is irreducible on ${E}$, by minimum principle (see for instance the book of Woess~\cite{Woess}), from this it follows that $f=0$ and consequently $h = h(0)V$. 
If  $\E_\cdot(\tau) < + \infty$, then by Theorem~\ref{th2}, the function $V$ is potential for $(Z(t))$ and consequently, the function $f= h - h(0)V\geq 0$ is non-trivial.

\section{Proof of Theorem~\ref{theorem_4}}\label{theorem_4_proof}
\subsection{Preliminary results} We begin the proof of this theorem with the following preliminary results. The following statement was proved in the paper Ignatiouk-Robert~\cite{Ignatiouk-ladder_heights} (see  Proposition~9.1 of \cite{Ignatiouk-ladder_heights}).

\begin{prop}\label{LD_estimates_pr} Under the hypotheses (B0)-(B3),  for any $q\in S_+$ and any sequence of points $(y_n)\in{E}^\N$, with $\lim_n\|y_n\| = +\infty$ and $\lim_n y_n/\|y_n\| = q \in{\cal C}\setminus\{0\}$, 
\[
\liminf_{n\to\infty} \frac{1}{\|y_n\|}\log G(0,y_n) ~=~ - \alpha(q) \cdot q, \quad \forall x\in{E}.
\]
\end{prop} 
We need the following consequence of this proposition. 
\begin{lemma}\label{LD_lem} Under the hypotheses (B0)-(B3),  for any $q\in S_+$ and any sequence of points $(y_n)\in{E}^\N$, with $\lim_n\|y_n\| = +\infty$ and $\lim_n y_n/\|y_n\| = q \in{\cal C}\setminus\{0\}$, uniformly on $q\in S_+$, 
\be\label{LD_e2} 
\liminf_{n\to\infty} \frac{1}{\|y_n\|}\log \P_0\bigl(Z(t) = y_n \; \text{for some $t\in\N$}\bigr) ~\geq~ - \alpha(q) \cdot q, \quad \forall x\in{E}.
\ee
\end{lemma} 
\begin{proof}Remark first of all that for any $y\in\E$, 
\begin{align*}
G(0,y_n) &=~ \P_0\bigl(Z(t) = y_n \; \text{for some $t\in\N$}\bigr)  G(y_n,y_n) \\ &\leq~ \P_0\bigl(Z(t) = y_n \; \text{for some $t\in\N$}\bigr)  G_h(0,0)
\end{align*} 
where $G_h(x,y)$ denotes the Green function of the homogeneous random walk $(X(t))$ on $\Z^d$ with transition probabilities $\P_x(X(1)=y) = \mu(y-x)$. When combined with Proposition~\ref{LD_estimates_pr}, this relation implies that for any $q\in S_+$ and any sequence of points $(y_n)\in{E}^\N$, with $\lim_n\|y_n\| = +\infty$ and $\lim_n y_n/\|y_n\| = q \in{\cal C}\setminus\{0\}$, \eqref{LD_e2} holds~: for any $\eps > 0$ there are $N(q,\eps) > 0$ and $\delta(q,\eps) > 0$ such that 
\[
\frac{1}{\|y_n\|}\log \P_0\bigl(Z(t) = y_n \; \text{for some $t\in\N$}\bigr) ~\geq~ - \alpha(q) \cdot q - \eps
\]
whenever $\|y_n\| \geq N(q,\eps)$ and 
\[
\left\| \frac{y_n}{\|y_n\|} - q \right\|  < \delta(q,\eps). 
\]
To complete the proof of our lemma it is therefore sufficient to show that this convergence is uniforme with respect to $q\in S_+$. Without any restriction o generality we can assume that for any $q\in S_+$, 
\be\label{LD_e3}
0 < \delta(q,\eps)   \leq \eps \left(  \sup_{\alpha \in \partial D} \|\alpha\|\right)^{-1}. 
\ee
The set $S_+$ being compact, there is a finite subset $\{q_1,\ldots,q_k\}\subset S_+$ such that 
\[
S_+ ~\subset~ \bigcup_{i=1}^k B\left(q_i, \frac{\delta(q_i,\eps)}{2} \right).
\]
where $B(q, \delta)$ denotes an open ball centered at $q$ and having a radius $\delta$. Hence, for any $q\in S_+$, there is $i\in\{1,\ldots,k\}$ such that $\|q_i - q\| < \delta(q_i,\eps/2)$ and consequently, 
letting 
\[
N_\eps = \max_{1\leq i \leq k}N(q_i,\eps) \quad \text{and} \quad \delta_\eps = \min_{1\leq i\leq k} \delta(q_i,\eps)/2, 
\]
for any  and $y\in E$ with $\|y\| \geq N_\eps$ and $\|q - y/\|y\|\| < \delta_\eps$, one gets $\|q_i - y/\|y\|\| < \delta(q_i,\eps)$, and consequently 
\begin{align*}
\frac{1}{\|y\|}\log \P_0\bigl(Z(t) = y \; \text{for some $t\in\N$}\bigr) &\geq~ - \alpha(q_i) \cdot q_i - \eps\\
&\geq - \alpha(q_i) \cdot q  - \|\alpha(q_i) \| \delta(q_i,\eps) - \eps.
\end{align*} 
Since by \eqref{LD_e3}, $\|\alpha(q_i) \| \delta(q_i,\eps)  \leq \eps$, 
and according to the definition of the mapping $q\to\alpha(q)$,
\[
\alpha(q_i)\cdot q ~\leq~\sup_{\alpha\in\partial D} \alpha\cdot q ~=~\alpha(q)\cdot q, 
\]
we conclude therefore that for any $\eps > 0$ there are $N_\eps > 0$ and $\delta_\eps > 0$ such that for any $y\in E$, 
\[
\frac{1}{\|y\|}\log \P_0\bigl(Z(t) = y \; \text{for some $t\in\N$}\bigr) ~\geq -\alpha(q)\cdot q - 2\eps 
\]
whenever $\|y\| \geq N_\eps$ and $\|q - y/\|y\|\| < \delta_\eps$. Lemma~\ref{LD_lem} is therefore proved. 
\end{proof} 

Recall that the function $q\to \alpha(q)$  was defined on the unit sphere $S$. We extend this function on $\R^d$ by letting $\alpha(0)=0$ and $\alpha(x)= \|x\| \alpha(x/\|x\|)$ for $x\not=0$. 

\begin{lemma}\label{dominated_function} Under the hypotheses (B0)-(B3),  for $\delta > 0$ small enough, 
\[
\sum_{y\in \Z^d} \mu(y) \exp\bigl(\alpha(x+y)\!\cdot\!(x+y) + \delta\|x+y\|\bigr) ~<~+\infty, \quad \forall x\in\Z^d. 
\]
\end{lemma} 
\begin{proof} Remark that for any $x,y\in\R^d$, according to the definition of the mapping $q\to\alpha(q)$, 
\[
\alpha(x+y) \cdot (x+y) ~=~ \sup_{\alpha\in D} \alpha\cdot (x+y) ~\leq~ \sup_{\alpha\in D} \alpha\cdot x + \sup_{\alpha\in D} \alpha\cdot y ~=~ \alpha(x)\cdot x + \alpha(y)\cdot y 
\]
and consequently, for any $x\in E$, 
\[
\sum_{y\in \Z^d} \mu(y) \exp\bigl(\alpha(x+y)\!\cdot\!(x+y) + \delta\|x+y\|\bigr) ~\leq~ C(x) \sum_{y\in \Z^d} \mu(y) \exp\bigl(\alpha(y)\cdot y + \delta\|y\|\bigr) 
\]
with 
\[
C(x) ~=~ \exp\bigl(\alpha(x)\cdot x + \delta\|x\|\bigr).
\]
To prove Lemma~\ref{dominated_function}, it is therefore sufficient to prove that for $\delta > 0$ small enough, 
\[
\sum_{y\in \Z^d} \mu(y) \exp\bigl(\alpha(y)\cdot y + \delta\|y\|\bigr) < + \infty. 
\]
Furthermore, recall that because of the assumption (B2), the step generating function 
\[
R(\alpha) = \sum_{y\in\Z^d} \mu(y) \exp(\alpha\cdot y) 
\]
is finite in a neighborhood of the set $D=\{\alpha :~R(\alpha) \leq 1\}$. Hence, for any $\alpha\in \partial D$, there is $\delta(\alpha) > 0$ such that
\[
\sum_{y\in\Z^d} \mu(y) \exp(\alpha\cdot y + \delta(\alpha) \|y\|) ~<~+\infty.
\]
The set $\partial D$ being compact, there is a finite subset $\{\alpha_1,\ldots,\alpha_k\}\subset\partial D$ such that 
\[
\partial D \subset \bigcup_{i=1}^k B(\alpha_i, \delta(\alpha_i)/2). 
\]
Letting 
\[
\delta_0 = \min_{1\leq i \leq k} \delta(\alpha_i),
\]
we conclude therefore that for any $\alpha\in\partial D$, there is $i\in\{1,\ldots,k\}$ such that 
\[
\|\alpha - \alpha_i\| < \delta(\alpha_i)/2
\]
and consequently, 
\begin{align*}
\sum_{y\in\Z^d} \mu(y) \exp(\alpha\cdot y + \delta_0 \|y\|) &\leq~ \sum_{y\in\Z^d} \mu(y) \exp\bigl(\alpha_i\cdot y + \bigl(\delta_0 + \delta(\alpha_i)/2\bigr)\|y\|\bigr) \\
&\leq~ \sum_{y\in\Z^d} \mu(y) \exp\bigl(\alpha_i\cdot y +  \delta(\alpha_i)\|y\|\bigr) ~<~+\infty. 
\end{align*} 
There is therefore $\delta_0 > 0$ such that 
\be\label{dominated_function_eq}
\sum_{y\in\Z^d} \mu(y) \exp(\alpha\cdot y + \delta_0 \|y\|) ~<~+\infty, \quad \forall \alpha\in\partial D. 
\ee
Furthermore, recall that under our assumptions, the function $q\to \alpha(q)$ is continuous on the unit sphere $S$. The unit sphere $S$ being compact,  the function $q\to\alpha(q)$ is therefore uniformly continuous on $S$ and consequently, there is $\sigma > 0$ such that for any $q,q'\in S$, 
\[
\|\alpha(q)-\alpha(q')\| < \delta_0/2  \quad \text{whenever} \quad \|q-q'\| < \sigma.
\]
Moreover, there is a finite subset $\{q_1,\ldots,q_k\}\subset S$ such that 
\[
S \subset \bigcup_{j=1}^k B(q_j, \sigma). 
\]
For any non-zero $y\in\R^d$, there is therefore $j\in\{1,\ldots, k\}$ such that $\| q_i - y/\|y\| \| < \sigma$ and 
\[
 \|\alpha(y) - \alpha(q)\|  = \|\alpha(y/\|y\|) - \alpha(q)\| < \delta_0/2.
\]
The last inequality shows that 
\[
\alpha(y)\cdot y - \alpha(q_j)\cdot y ~<~ \delta_0\|y\|/2,
\]
and consequently, 
\[
\exp(\alpha(y)\cdot y + \delta_0 \|y\|/2) ~\leq~ \sum_{j=1}^k \exp(\alpha(q_j)\cdot y + \delta_0 \|y\|).
\]
When combined with \eqref{dominated_function_eq}, this relation prove that 
\[
\sum_{y\in\Z^d} \mu(y) \exp(\alpha(y)\cdot y + \delta_0 \|y\|/2) ~\leq~ \sum_{j=1}^k\sum_{y\in\Z^d} \mu(y) \exp(\alpha(q_j)\cdot y + \delta_0 \|y\|) ~<~+\infty. 
\]
\end{proof} 

Now we are ready to prove the following statement. 
\begin{lemma}\label{harmonic_limits}  If a sequence $(y_n)\in {E}$ converges in the Martin compactification $E_M$ of $E$ to some point $\eta\in\partial_M E$, then the limit function 
\[
K(z,\eta) ~=~ \lim_n K(z,y_n), \quad z\in E
\]
is harmonic for $(Z(t))$. 
\end{lemma} 
\begin{proof} By Lemma~\ref{LD_estimates_pr} for any $\delta > 0$ there is $N>0$ such that for any $x\in E$ with $\|x\| \geq N$, 
\[
\P_0\bigl(Z(t) = y_n \; \text{for some $t\in\N$}\bigr)  \geq \exp( - a(x)\cdot x  - \delta \|x\|).
\]
Since for any $x,y\in E$,
\[
\P_0\bigl(Z(t) = y_n \; \text{for some $t\in\N$}\bigr)  G(x,y) ~\leq~G(0,y) 
\]
this proves that for any $\delta > 0$ there is $C>0$ such that 
\[
K(x,y_n) ~=~ \frac{G(x,y_n)}{G(0,y_n)} ~\leq~C \exp( a(x)\cdot x  + \delta \|x\|), \quad \forall x\in E, \; n\in\N. 
\]
Remark now that for any $n\in\N$ and $x\in E\setminus\{y_n\}$, 
\[
\E_x\bigl( K(Z(1),y_n)\bigr) = K(x,y_n), 
\]
and recall that by Lemma~\ref{dominated_function}, the random variable  $\exp( a(Z(1))\cdot Z(1)  + \delta \|Z(1)\|)$ is $P_x$ integrable for $\delta > 0$ small enough.  Hence, using  dominated convergence theorem we conclude  that 
\[
\E_x\bigl( K(Z(1),\eta)\bigr) = K(x,\eta).
\]
\end{proof}

\subsection{Proof of Theorem~\ref{theorem_4}.} Now we are ready to complete the proof of Theorem~\ref{theorem_4} and we begin our analysis with a particular case when the mean step of the homogeneous random walk 
\[
m = \sum_{x\in\Z^d} x\mu(x) 
\]
belongs to the cone ${\cal C}$, and $q = m/\|m\| \in S_+^\infty$. Recall that in this case, 
\[
\alpha(q) ~=~\alpha(m) ~=~ 0
\]
because $q(0) = \nabla R(0)/\|\nabla R(0)\| = m/\|m\|$ and according to the definition of $S_+^\infty$, 
\[
\E_\cdot(\tau) = +\infty. 
\]
Let a sequence of points $(y_n)\in E^\N$ be such that  $\lim_n \|y_n\| = +\infty$ and $\lim_n y_n/\|y_n\| = m/\|m\|$. If a subsequence $(y_{n_k})$ converges in the Martin compactification to some point $\eta$, then by Lemma~\ref{harmonic_limits}, the limit function 
\[
K(x,\eta) ~=~ \lim_{k\to\infty} K(x,y_{n_k}), \quad x\in E,
\]
is harmonic for $(Z(t))$. Moreover,   by Proposition~\ref{LD_estimates_pr}, 
\[
\liminf_{n\to\infty} \frac{1}{\|y_n\|}\log G(0,y_n) ~\geq~ 0 \quad \forall x\in{E}.
\]
and consequently, by Corollary~\ref{main_result_cor2} and according to the definition of the functions $k_q(\cdot)$, 
\[
K(x,\eta) = V(x) = \exp(-\alpha(m/\|m\|)\cdot x) k_{m/\|m\|}(x) = k_{m/\|m\|}(x), \quad \forall x\in E. 
\]
Since the limit function $K(\cdot,\eta)$ does not depend on the convergent subsequence $(y_{n_k})$, this implies that 
\[
\lim_{n\to\infty} K(x,y_n) = k_{m/\|m\|}(x), \quad \forall x\in E.
\]
For $q= m/\|m\| \in S_+^\infty$,  the function $k_{m/\|m\|}$ is  therefore harmonic fo $(Z(t))$ and \eqref{theorem_4_e1} holds.

To extend this result  for an arbitrary $q\in S_+^\infty$ we use classical Cramer's transform. Remark that under the hypotheses (B1)-(B3), for any $\alpha\in \partial_+D$, the twisted random walks $(X_\alpha(t))$ and  $(Z_\alpha(t))$  satisfy the  conditions similar to (B1)-(B3)~: 
\begin{enumerate}
\item[(B1)]  The random walk $(Z_\alpha(n))$ is transient  on $E~=~ {\cal C}\cap\Z^d$ and satisfies the following communication condition~:  there are $\kappa_0 >0$ and a finite set ${\cal E}_0 \subset \Z^d$ such that 
\begin{itemize}
\item[(a)] $\mu_\alpha(x) \dot= \exp(\alpha\cdot x) \mu(x) > 0$ for all $x\in{\cal E}_0$; 
\item[(b)] for any $x\not= y$, $x,y\in E$ there exists a sequence $x_0, x_1, \ldots , x_n\in E$ with $x_0=x$, $x_n=y$ and $n\leq \kappa_0 |y-x|$ such that $x_j-x_{j-1}\in {\cal E}_0$ for all $j\in\{1,\ldots,n\}$. 
\end{itemize} 
\item[(B2)] the step generating function of the homogeneous random walk $(X_\alpha(t))$
\[
R_\alpha(\beta) \dot= \sum_{x\in \Z^d} \exp(\beta\cdot x) \mu_\alpha(x) ~=~ R(\alpha + \beta) 
\]
is finite in a neighborhood of the set $D_\alpha~=~\{\beta\in\R^d~:~ R_\alpha(\beta) \leq 1\} ~=~ D - \alpha$;
\item[(B3)] the mean step $m_\alpha = \E_0(X_\alpha(1))$ of the homogeneous random walk $(X_\alpha(t))$ is non-zero~: 
\[
m_\alpha = \sum_{x\in\Z^d} x \exp(\alpha\cdot(y-x))\mu(y-x) ~=~ \nabla R(\alpha) ~=~q(\alpha) \|\nabla R(\alpha)\| \not= 0.
\]
\end{enumerate} 
Moreover, 
\[
 q = \nabla R(\alpha(q))/\|\nabla R(\alpha(q))\| = m_{\alpha(q)}/\|m_{\alpha(q)}\| 
\]
and hence for any $q\in S_+^\infty$, the above arguments applied fo the twisted random walk $(Z_{\alpha(q)}(t))$ prove that for any sequence of points $(y_n)\in E^\N$ with $\lim_n \|y_n\| = +\infty$ and $\lim_n y_n/\|y_n\| = q$, the sequence of functions 
\[
K_{\alpha(q)}(\cdot,y_n) = \frac{G_{\alpha(q)}(\cdot,y_n)}{G_{\alpha(q)}(0,y_n)}
\]
where $G_{\alpha(q)}(x,y)$ denotes the Green function of $(Z_{\alpha(q)}(t))$, converges point-wise to the renewal function $V_{\alpha(q)}(x)$ of the corresponding ladder height process $(H_{\alpha(q)}(n))$.  Since for any $x,y\in E$,  
\begin{align*}
G_{\alpha(q)}(x,y) &=~  \sum_{t=0}^\infty  \P_x(Z_{\alpha(q)}(t) = y)  = \sum_{t=0}^\infty \exp(\alpha(q)\cdot(y-x)) \P_x(Z(t) = y) \\ &=~ \exp(\alpha(q)\cdot(y-x))G(x,y)
\end{align*} 
we conclude therefore that the sequence of functions 
\begin{align*}
K(x,y_n) &= \frac{G(x,y_n)}{G(0,y_n)} = \exp(\alpha(q)\cdot x) \frac{G_{\alpha(q)}(x,y_n)}{G_{\alpha(q)}(0,y_n)} ~=~ \exp(\alpha(q)\cdot x) K_{\alpha(q)}(x,y_n)
\end{align*} 
converge point-wise to the function 
\[
k_q(x) = \exp(\alpha(q)\cdot x)  V_{\alpha(q)}(x), \quad x\in E,  
\]
and by Lemma~\ref{harmonic_limits},  the limit function $k_q$ is harmonic for $(Z(t))$. Theorem~\ref{theorem_4} is therefore proved.

\section{Proof of Theorem~\ref{main_result_example_th1}.}\label{th1_example_proof} 

\subsection{Main ideas of the proof.} Before proving  Theorem~\ref{main_result_example_th1} in a general case, let us notice that in a particular case,  when the boundary of the cone ${\cal C}$ is a hyperplane in $\R^d$, i.e. when for some non-zero vector $\gamma\in\R^d$, 
\[
{\cal C} = \{x\in\R^d :~x\cdot \gamma \geq 0\},
\]
and in particular when $d=1$ and ${\cal C} = [0,+\infty[$, this is a simple consequence of classical results concerning one dimensional random walks. Indeed, in this case, $X(t)\cdot \gamma$ is a random walk in $\R$ with the mean 
$m\cdot \gamma$, and $\tau=\inf\{t\geq 0:~X(t)\not\in{\cal C}\}$ is the first time when the random walk $X(t)\cdot \gamma$ become negative. If $m\in{\cal C} = \{x\in\R^d :~x\cdot \gamma \geq 0\}$, then clearly $m\cdot \gamma \geq 0$, and consequently,  the stopping time $\tau$ is not integrable (see for instance the books of Spitzer~\cite{Spitzer} and Feller~\cite{Feller}).

Remark moreover that in the case when the mean step $m$ belongs to the interior of the cone ${\cal C}$, by the strong law of large numbers $\P_x(\tau = +\infty) > 0$ and consequently $\E_x(\tau) = +\infty$ for some $x\in E = {\cal C}\cap\Z^d$ (see for instance, the proof of Lemma 3.7 in the paper of Duraj~\cite{JDuraj}).  Under the hypotheses (B1), from this it follows that, $\P_x(\tau = +\infty) > 0$ and $\E_x(\tau) = +\infty$  for all  $x\in E = {\cal C}\cap\Z^d$. Hence, to prove Theorem~\ref{main_result_example_th1} it is sufficient to consider the case when the mean step $m$ belongs to the boundary $\partial{\cal C}$ of the cone ${\cal C}$.  

Remark finally that because of Assumption~(B1), the random walk $(X(t))$ is irreducible in $\Z^d$, and consequently, the covariance matrix $\Gamma$ of the steps of the random walk $(X(t))$ is non degenerate~:~for any $u\in\R^d$, 
\[
u \cdot \Gamma u ~=~ \E\left( \bigl((X(1)-m)\cdot u\bigr)^2\right) ~>~0,
\]
This proves that there is an invertible  matrix $M$ for which the steps of the random walk $(\hat{X}(t) = MX(t))$ in the lattice $M\Z^d$ have the identity covariance matrix~:
\[
 \E\left( \bigl((X(1)-m)\cdot u\bigr)^2\right)  ~=~ u\cdot u, \quad \forall u\in\R^d
\]
and the mean 
\[
\hat{m} = M m. 
\]
To prove Theorem~\ref{main_result_example_th1} it is  therefore sufficient to show that the first time $\hat\tau$ when the random walk $(\hat{X}(t))$ exits form the cone $\hat{C} ~=~ M{\cal C}$ is non integrable 
whenever $\hat{m} \in \partial \hat{\cal C}$. 
\medskip 

From now on we assume that $d\geq 2$ and that $m\in\partial{\cal C}$, or equivalently, that $\hat{m} \in\partial\hat{\cal C}$.

Denote by $\hat\Pi$ the hyperplane of $\R^d$ which is orthogonal to the vector $\hat{m}$ :
\[
\hat\Pi = \{x\in\R^d :~x\cdot \hat{m} = 0\}
\]
and let 
\[
{\cal C}_\eps ~=~ \{x\in\R^d~:~ 0 \leq \angle(x,\hat{m}) < \eps\} 
\]
where $\angle(x,\hat{m})$ denotes the angle between the vectors $x$ and $\hat{m}$. 
For $x\in\R^d$,  $Pr(x)$ will denote the orthogonal projection of $x$ on $\hat\Pi$.

The projection $\hat{S}(t) = Pr(\hat{Z}(t))$ onto $\hat\Pi$ of the random walk $(\hat{Z}(t))$ is then a $d-1$ dimensional centered random walk satisfying the hypotheses of Denisov and Wachtel~\cite{Denisov-Wachtel:2}. Using the results of this paper, for any convex cone ${\cal K}\subset \hat{\Pi}$,  on can get  the exact asymptotic of of the tail probability for the first time 
when the random walk $(\hat{Z}(t))$ exits from the cone 
\[
\{x\in\R^d~:~Pr(x)\in {\cal K}\}.
\]
To prove Theorem~\ref{main_result_example_th1}, we will construct a convex cone $\hat{\cal K}\subset \hat{\Pi}$ such that 
\begin{itemize}
\item[--] for any $\eps > 0$, the first time when the random walk  $(\hat{Z}(t))$ exits from the cone 
\[
\{x\in\R^d~:~Pr(x)\in \hat{\cal K}\}\cap {\cal C}_\eps 
\]
is non-integrable, and 
\item[--] for $\eps > 0$ small enough, $\{x\in\R^d~:~Pr(x)\in \hat{\cal K}\}\cap {\cal C}_\eps \subset \hat{\cal C}$. 
\end{itemize} 
Before proving our theorem, we recall the results of  Denisov and Wachtel~\cite{Denisov-Wachtel:2} that we need for our proof.

\subsection{Existing preliminary results for centered random walk.} 
Let ${\cal K}$ denote an open convex cone in $\R^k$ with a vertex at the origin $0\in\R^k$.  Consider a random walk $(S(t))$  in $\R^k$ with steps $\xi_i, \; i\in\N$~: 
\[
S(t) = \sum_{i=1}^t \xi_i, \quad t\in\N,
\]
where  $\xi_i, \; i\in\N$,  are centered, independent and identically distributed  random variables valued in $\R^k$, and let $\tau_{\cal K}$ denote the first time when the random walk $(S(t))$ exists from the set  ${\cal K}$~: 
\[
\tau_{\cal K} ~=~ \inf\{t > 0~: S(t)\not\in {\cal K} \}.
\]
In their paper~\cite{Denisov-Wachtel:2}, Denisov and Wachtel assumed that 

\noindent
{\bf Assumption (C1)} the cone ${\cal K}$ is either convex or star-like and $C^2$;

\noindent
{\bf Assumption (C2)} the random vectors $\xi_i$ are centered ( $\E(\xi_i) = 0$)  and reduced (e.i. that the random vector $\xi_i$ has the identity covariance matrix); 

\noindent
{\bf Assumption (C3)} $\E(\|\xi_i\|^{p}) < +\infty$ with $p=p^*$ if $p^* > 2$ and with some $p > p^*$ if $p^*\leq 2$.

\medskip 
\noindent
Under these hypotheses,  for a copy of the  random walk $(S(t))$ killed upon the time $\tau_{\cal K}$, Denisov and Wachtel constructed a nontrivial harmonic function ${\cal V} : \R^k\to\R_+$, 
such that 
\[
{\cal V}(x) = 0 \quad \quad \forall x\in\R^k\setminus{\cal K}
\]
and 
\[
{\cal V}(x) ~\leq~ C(\|x\|^{p^*} + 1), \quad \quad \forall x\in {\cal K}, 
\]
and proved that for any $x\in{\cal K}$, 
\be\label{e0}
\lim_{t\to+\infty} t^{p^*/2} \P_x(\tau_{\cal K} > t) ~=~ \kappa {\cal V}(x)
\ee
with some absolute constants  $\kappa > 0$  and $C> 0$ (see Theorem~1 and Lemma~14 in \cite{Denisov-Wachtel:2}). The function ${\cal V}$ and the constant  $p^* > 0$ were defined in terms of the  minimal (up to a constant) and strictly positive on ${\cal K}$ solution of the boundary problem~: 
\[
\Delta u(x) = 0, \; x\in {\cal K} \quad \text{and} \quad \left. u\right|_{\partial K} ~=~0.
\]
If $k=1$ then there is only one non-trivial cone ${\cal K} = ]0,+\infty[$ and in this case $u(x) = x$ for all $x \geq 0$, and $p^*=1$. If $k \geq 2$, the number $p^*$  can be found as follows~: 

Let $L$ be the Laplace-Beltrami operator  on the unit sphere ${\mathbb S}^{k-1}$ in $\R^{k}$. If $\Sigma = {\cal K}\cap\mathbb{S}^{k-1}$ is regular with respect to $L$,  then there exists a complete set of orthogonal eigenfunctions $w_j$ of $L$  satisfying 
\[
L w_i(x) = - \lambda_j w_j(x), \quad \quad x\in \Sigma,
\]
and
\[
w_j(x) = 0, \quad x\in\partial\Sigma,
\]
such that 
\[
0< \lambda_1 < \lambda_2 < \ldots
\]
and 
\[
u(x) ~=~ \|x\|^{p^*} w_1\left(\frac{x}{\|x\|}\right) 
\]  
with 
\[
p^* ~=~\sqrt{\lambda_1 + (k/2)^2} - (k/2 - 1).
\]

\medskip 
\noindent
Remark moreover that in a particular case, when the cone ${\cal K}$ is convex, from \eqref{e0} it follows 

\begin{prop}\label{D-W_cor} If the conditions (C2) and (C3) are satisfied and the cone ${\cal K}$ is convex, then there is $x_0\in{\cal K}$ such that 
\[
\inf_{y\in x_0+{\cal K}} {\cal V}(y) \geq {\cal V}(x_0) ~>~0. 
\]
\end{prop} 
\begin{proof} Indeed, the function ${\cal V}:{\cal K}\to\R_+$ being non-trivial, there is $x_0\in{\cal K}$ such that ${\cal V}(x_0)\not=0$. If the cone ${\cal K}$ is convex, then for any $y\in{\cal K}$, $y+{\cal K}\subset {\cal K}$ and consequently,  $P_{x_0+y}$-a.s.
\[
\tau_{\cal K} \geq \tau_{y+{\cal K}},
\]
where $\tau_{y+{\cal K}} ~=~ \inf\{t > 0:~ S(t)\not\in y+{\cal K}\}$. 
Since by homogeneity, for any $t\in\N$ 
\[
\P_{x_0+y}(\tau_{y+{\cal K}}  > t) ~=~ \P_{x_0}(\tau_{\cal K} > t), 
\]
from this it follows that 
\[
\P_{x_0+y}(\tau_{\cal K}  > t) ~\geq~ \P_{x_0}(\tau_{\cal K} > t), \quad \forall y\in {\cal K},
\]
and hence, by \eqref{e0}, 
\[
{\cal V}(x) \geq {\cal V}(x_0) > 0, \quad \forall x\in x_0+{\cal K}. 
\]
\end{proof} 
We will use these results for a circular cone ${\cal K}(\theta,v) \dot= \{x\in\R^k ~:~ x\not=0, \; \angle(x, v) < \theta\}$
for some non zero vector $v\in\R^k$ and $0 < \theta <\pi$, where $\angle(x,v)$ denotes the angle between the vectors $x,v\in\R^k$. In this particular case, the function $w_1$ and the constant $p^*$ can be represented in the following way~(see Burkholder~\cite{Burkholder} p.193, and the references therein)~: 

\noindent 
1) If $k=2$ then 
\[
u(x) = \|x\|^{p^*}\cos\bigl(p^*\angle(x,v)\bigr) \quad \text{with} \quad p^* = \pi/(2\theta)
\]
2) If  $k > 2$, then   for  $p > 0$, 
\begin{itemize} 
\item[--] the hypergeometric function
\[
F(a,b,c,t) ~\dot=~ \sum_{j=0}^\infty \frac{(a)_j(b)_j}{(c)_j j!} t^j
\]
with $a=-p$, $b=p+k-2$, ans $c= (k-1)/2$, where $(a)_0= 1$, $(a)_1 = a$, $(a)_2 = a(a+1)$, \ldots, is well defined for $|t| < 1$;
\item[--] the function 
\[
\theta \to h(\theta) = F\Bigl(a,b,c, \bigl(1- \cos(\theta)\bigr)/2\Bigr)  
\] 
is well defined in $[0,\pi[$ with  $h(0)=1$ and has in the interval $[0,\pi[$ at least one zero. 
\end{itemize} 
Let $\theta_k(p)$ denote the smallest zero of $h$ in $[0,\pi[$. Then the mapping $p\to \theta_k(p)$ is continuous and strictly decreasing from $]0,+\infty[$ to $]0,\pi[$ with $\theta_k(1) ~=~\pi/2$, the inverse mapping $\theta \to p_k(\theta)$ is continuous and strictly decreasing from $]0,\pi[$ to $]0,+\infty[$ with $p_k(\pi/2) = 1$, and 
\[
u(x) = \|x\|^{p^*} h(\angle(x,v)), \quad x\in{\cal K}(\theta^*,v),  
\]
with $p^* = p_k(\theta^*)$. 
Since for any non-zero vector $v\in\R^k$ and $0 < \theta < \pi/2$, the circular cone ${\cal K}(\theta,v)$ is convex, using  the results of Denisov and Wachtel~\cite{Denisov-Wachtel:2} and Proposition\ref{D-W_cor} one gets 

\begin{prop}\label{Denisov_Wachtel} Suppose that the condition (C2) is satisfied and let $\E(\|\xi_k\|^{2+\delta}) < +\infty$ for some $\delta > 0$. Then  the following assertions hold: \\ 
1) there is   $0 < \theta_2  < \pi/2$ such that  $p^*_k(\theta_2) = 2$; \\
2) for any non-zero vector $v\in\R^k$,  there is a non-trivial  function ${\cal V}_v: {\cal K}(\theta_2,v)\to \R_+$ such that $\bigl({\cal V}_v(S(n\wedge \tau_{{\cal K}(\theta_2,v)})\bigr)$ is a martingale relative to the natural filtration of $(S(t))$,
\[
{\cal V}(x) = 0 \quad \quad \forall x\in\R^k\setminus{\cal K}(\theta_2,v)
\]
and 
\[
{\cal V}_v(x) ~\leq~ C(\|x\|^{2} + 1), \quad \quad \forall x\in {\cal K}(\theta_2,v);
\]
3) for any $x\in{\cal K}(\theta_2, v)$,
\[
\lim_{t\to\infty} t\,  \P_x(\tau_{{\cal K}(\theta_2,v)} > t) ~=~ \kappa {\cal V}_v(x).
\]
4) for some $x_0\in{\cal K}(\theta_2,v)$, 
\[
\inf_{y\in x_0+{\cal K}(\theta_2,v)} {\cal V}(y) \geq {\cal V}(x_0) ~>~0. 
\]
\end{prop} 
Remark that for those $x\in{\cal K}(\theta_2,v)$ for which ${\cal V}_v(x)\not=0$, from the last assertion of this proposition it follows that $\E_x\left(\tau_{{\cal K}(\theta_2,v)}\right) ~=~+\infty$. 

To prove Theorem~\ref{main_result_example_th1}, we will use the following consequence of this result. 

\begin{prop}\label{prop5_2} Let $\tau^*$ be a stopping time relative to some filtration $({\cal F}_k)$. Assume that the hypotheses of Proposition~\ref{Denisov_Wachtel} are satisfied and let for any $n\in\N$, the random variables $\xi_1,\ldots,\xi_n$ be ${\cal F}_n$-measurable and the random variable $\xi_{n+1}$  is independent on ${\cal F}_n$. Then for any  $v\in\R^k$ and $x\in {\cal K}(\theta_2,v)$,  
\[
\E_x\left(\tau^*\wedge\tau_{{\cal K}(\theta_2,v)}\right) ~=~+\infty \quad \text{whenever} \quad \E_x \left({\cal V}_v(S\bigl(\tau^*\wedge\tau_{{\cal K}(\theta_2,v)}\bigr)\right) ~<~{\cal V}_v(x). 
\]
\end{prop} 
\begin{proof} By the stopping time theorem, from the second assertion of Proposition~\ref{Denisov_Wachtel} it follows that the  sequence 
\[
{\cal V}_v\bigl(S(t\wedge\tau^*\wedge \tau_{{\cal K}(\theta_2,v)})\bigr), \quad t\in\N, 
\]
is a non-negative  $({\cal F}_t)$ - martingale. Since for any $x\in{\cal K}(\theta_2,v)$, $\P_x$-a.s. $\tau_{{\cal K}(\theta_2,v)} < +\infty$, it converges a.s. to 
\[
{\cal V}_v\bigl(S(\tau^*\wedge \tau_{{\cal K}(\theta_2,v)})\bigr). 
\]
The main idea of the proof of Proposition~\ref{prop5_2} is the following~: Assuming that 
\[
\E_x\left(\tau^*\wedge\tau_{{\cal K}(\theta_2,v)}\right) ~<~+\infty,
\]
we will prove that  the martingale ${\cal V}_v\bigl(S(t\wedge\tau^*\wedge \tau_{{\cal K}(\theta_2,v)})\bigr)$ is uniformly integrable and we will conclude that 
\[
{\cal V}_v(x)  ~=~ \E_x \left({\cal V}_v\bigl(S(\tau^*\wedge\tau_{{\cal K}(\theta_2,v)})\bigr)\right). 
\]
For this  we first notice that  a sequence 
\[
{\cal M}_1(t)  ~=~ \sum_{i=1}^k |S_i(t)|, \quad t\in\N,
\]
where  $S_i(t)$ denotes the $i$-th coordinate of $S(t)$,  is a nonnegative $({\cal F}_t)$ - submartingale.  By the stopping time theorem, from this it follows that the sequence 
\[
{\cal M}_1(t\wedge\tau^*\wedge\tau_{{\cal K}(\theta_2,v)}), \quad t\in\N, 
\]
is also a nonnegative $({\cal F}_t)$-submartingale, and by  Doob's $L^p$-inequality (see for instance the book of David Williams~\cite{DWilliams}) with $p=2$,
\[
\E_x\left(\sup_{1\leq s \leq t} {\cal M}_1^2(s\wedge\tau^*\wedge\tau_{{\cal K}(\theta_2,v)})\right) ~\leq~ 4 \sup_{1\leq s\leq t} \E_x \left({\cal M}_1^2(s\wedge\tau^*\wedge\tau_{{\cal K}(\theta_2,v)})\right), \quad \forall t\in\N. 
\]
Since for any $t\in\N$, 
\[Ò
 \|S(t)\|^2 \leq {\cal M}_1^2(t) ~\leq~ k \|S(t)\|^2, 
\]
and the sequence $\|S(t)\|^2= S(t)\cdot S(t) = \sum_{i=1}^nS_i^2(t)$ is also a $({\cal F}_t)$ - sub-martingale, from this it follows that 
\begin{align} 
\E_x\left(\sup_{1\leq s \leq t} \|S(s\wedge\tau^*\wedge\tau_{{\cal K}(\theta_2,v)})\|^2\right) &\leq~   \sup_{1\leq s\leq t} \E_x \left(\sup_{1\leq s \leq t} {\cal M}_1^2(s\wedge\tau^*\wedge\tau_{{\cal K}(\theta_2,v)})\right) \nonumber \\ 
&\leq~ 4  \sup_{1\leq s\leq t} \E_x \left( {\cal M}_1^2(s\wedge\tau^*\wedge\tau_{{\cal K}(\theta_2,v)})\right) \nonumber \\ 
&\leq~ 4 k \sup_{1\leq s\leq t} \E_x \left(\|S(s\wedge\tau^*\wedge\tau_{{\cal K}(\theta_2,v)})\|^2\right) \nonumber \\ 
&\leq~ 4k \E_x \left(\|S(t\wedge\tau^*\wedge\tau_{{\cal K}(\theta_2,v)})\|^2\right), \quad \forall t\in\N.  \label{e1} 
\end{align} 
Remark moreover that the sequence
\[
{\cal M}_2(t) ~=~ \|S(t)\|^2 - k t, \quad t\in\N,
\]
is $({\cal F}_t)$ - martingale. By the stopping time theorem, the sequence 
\[
{\cal M}_2(t\wedge\tau^*\wedge\tau_{{\cal K}(\theta_2,v)}) ~=~ \|S(t\wedge\tau^*\wedge\tau_{{\cal K}(\theta_2,v)})\|^2 - k \left( t\wedge\tau^*\wedge\tau_{{\cal K}(\theta_2,v)}\right) , \quad t\in\N, 
\]
is therefore also a  $({\cal F}_t)$ - martingale, and consequently, 
\be\label{e2} 
E_x \left(\|S(t\wedge\tau^*\wedge\tau_{{\cal K}(\theta_2,v)})\|^2\right) ~=~ \|x\|^2 + k\, \E_x(t\wedge\tau^*\wedge\tau_{{\cal K}(\theta_2,v)}), \quad \forall t\in\N. 
\ee
Remark finally that by Proposition~\ref{Denisov_Wachtel}, for any $\in\N$, 
\be\label{e3}
\sup_{1\leq s \leq t} {\cal V}(s\wedge \tau^*\wedge\tau_{{\cal K}(\theta_2,v)})  ~\leq~ C\left( \sup_{1\leq s\leq t} \|S(s\wedge \tau^*\wedge\tau_{{\cal K}(\theta_2,v)})\|^2 + 1\right).
\ee
When combined together, relations \eqref{e1}, \eqref{e2}  and \eqref{e3} show that for any $x\in{\cal K}(\theta_2,v)$ and $t\in\N$, 
\begin{align*} 
\E_x\left(\sup_{1\leq s \leq t}  {\cal V}(s\wedge \tau^*\wedge\tau_{{\cal K}(\theta_2,v)})   \right) &\leq~
C \E_x\left(\sup_{1\leq s \leq t}  \|S(s\wedge\tau^*\wedge\tau_{{\cal K}(\theta_2,v)})\|^2\   \right) + C \\ &\leq 4 C k \left(\|x\|^2 + k \E_x(t\wedge\tau^*\wedge\tau_{{\cal K}(\theta_2,v)})\right) + C
\end{align*} 
By the monotone convergence theorem, from the last relation it follows that 
\[
\E_x\left(\sup_{1\leq s < +\infty}  {\cal V}(s\wedge \tau^*\wedge\tau_{{\cal K}(\theta_2,v)})   \right) ~\leq~ 4 C n \left(\|x\|^2 + k \E_x(\tau^*\wedge\tau_{{\cal K}(\theta_2,v)})\right) + C.
\]
Whenever $\E_x\left(\tau^*\wedge\tau_{{\cal K}(\theta_2,v)})\right) < +\infty$, the martingale $\left({\cal V}_v\bigl(S(t\wedge\tau^*\wedge \tau_{{\cal K}(\theta_2,v)})\bigr)\right)$ 
is therefore uniformly integrable and 
\[
\E_x\left( {\cal V}_v\bigl(S(\tau^*\wedge \tau_{{\cal K}(\theta_2,v)})\bigr) \right) ~=~ \lim_{t\to\infty} \E_x\left( {\cal V}_v\bigl(S(t\wedge\tau^*\wedge \tau_{{\cal K}(\theta_2,v)})\bigr)\right) ~=~  {\cal V}(x)
\]
\end{proof} 

\subsection{Proof of Theorem~\ref{main_result_example_th1}.}

Let $\hat{v}$  be a unit vector in $\R^d$ which is normal  to $\partial\hat{\cal C}$ at the point $\hat{m}$ and such that 
\[
(\alpha - \hat{m}) \cdot \hat{v} ~\geq~0, \quad \forall \alpha\in\hat{\cal C}. 
\]
Remark that such a vector $\hat{v}$ exists because the cone $\hat{\cal C}$ is convex,  it is unique because the boundary of $\hat{\cal C}$ is $C^1$, and it  belongs to the hyperplane $\hat\Pi$ because the vector $\hat{m}\in\partial\hat{\cal C}$ is orthogonal to $\hat{v}$. We consider a circular cone in $\hat\Pi$~: 
\[
{\cal K}(\theta, \hat{v}) ~\dot=~ \{x\in\hat\Pi~:~ 0\leq \angle(x,\hat{v}) < \theta\},
\]
and we let 
\[
{\cal C}(\theta, \hat{v}) ~=~\{ x\in\R^d~:~ Pr(x) \in{\cal K}(\theta, \hat{v})\}, 
\]
where $Pr(x)$ denotes the orthogonal projection of $x$ onto $\hat\Pi$. For a given $0 < \eps < \pi/2$, we denote by  $\tau_\eps$  the  first time when the homogeneous random walk $(\hat{X}(t))$ in $M\Z^d$ exits from the circular cone 
\[
{\cal C}_\eps ~=~ \{x\in\R^d~:~ 0 \leq \angle(x,\hat{m}) < \eps\}, 
\]
\[
\tau_\eps ~=~ \inf\{t > 0~:~ \hat{X}(t)\not\in {\cal C}_\eps\}. 
\]
The first time when the random walk $(\hat{X}(t))$ exits from the set $\hat{\cal C}(\theta, \hat{v})$ will be denoted by $\tau_{{\cal K}(\theta,\hat{v})}$~: 
\[
\tau_{{\cal K}(\theta,\hat{v})} ~=~\inf\{t > 0~:~ \hat{X}(t)\not\in \hat{\cal C}(\theta, \hat{v})\}. 
\]
Remark that according to the definition of the cone $\hat{\cal C}(\theta, \hat{v})$, $\tau_{{\cal K}(\theta,\hat{v})}$ is also the first time when the centered random walk $\bigl(\hat{S}(t) = Pr(\hat{X}(t))\bigr)$ exits form the circular cone ${\cal K}(\theta, \hat{v})$. 

Under the hypotheses of Theorem~~\ref{main_result_example_th1}, the centered random walk $\bigl(\hat{S}(t) = Pr(\hat{X}(t))\bigr)$ satisfies the conditions of Proposition~\ref{Denisov_Wachtel}  and hence, there is   $0 < \theta_2  < \pi/2$ and a non-zero  function ${\cal V}: {\cal K}(\theta_2,\hat{v})\to \R_+$ such that 
 \[
{\cal V}(x) = 0 \quad \quad \forall x\in\Pi\setminus{\cal K}(\theta_2,\hat{v}),
\]
\be\label{e3p}
{\cal V}(x) ~\leq~ C(\|x\|^{2} + 1), \quad \quad \forall x\in {\cal K}(\theta_2,\hat{v});
\ee
and $\bigl({\cal V}(\hat{S}(t\wedge \tau_{{\cal K}(\theta_2,\hat{v})})\bigr)$ is a martingale relative to the natural filtration of $(\hat{S}(t))$ (and consequently also relative to the natural filtration of $(\hat{X}(t))$). 
Moreover,  letting for $x\in\R^d$, 
\[
{\cal V}(x) ~=~{\cal V}(Pr(x)), 
\]
and using Proposition~\ref{prop5_2}, we obtain that for any $x\in {\cal C}_\eps\cap{\cal C}(\theta, \hat{v}) \cap M\Z^d$, 
\be\label{e4}
\E_x\left(\tau_\eps\wedge\tau_{{\cal K}(\theta_2,\hat{v})}\right) ~=~+\infty \quad \text{whenever} \quad \E_x \left({\cal V}(\hat{X}\bigl(\hat\tau_\eps\wedge\tau_{{\cal K}(\theta_2,\hat{v})}\bigr)\right) ~<~{\cal V}(x). 
\ee
To prove  Theorem~~\ref{main_result_example_th1} we will choose $\eps > 0$ such that ${\cal C}_\eps\cap{\cal C}(\theta_2,\hat{v}) \subset{\cal C}$ and next we will show that for some $x\in {\cal C}_\eps\cap{\cal C}(\theta_2,\hat{v})\cap M\Z^d$, 
\be\label{e5}
\E_x \left({\cal V}\bigl(\hat{X}\bigl(\tau_\eps\wedge\tau_{{\cal K}(\theta_2,\hat{v})}\bigr)\bigr)\right) ~<~{\cal V}(x).
\ee
The last relation combined with  \eqref{e4} will imply  that 
\[
\E_x(\tau) ~\geq~ E_x\left(\tau_\eps\wedge\tau_{{\cal K}(\theta_2,\hat{v})}\right) ~=~+\infty. 
\]

\medskip 

To prove \eqref{e5} we need the following preliminary results. 

\begin{lemma}\label{example_lemma1} For any $\sigma > 0$ there are two strictly positive real numbers $\theta > 0$ and $C > 0$ such that for any $t\in\N$, 
\be\label{exemple_lemma1_e1} 
\P_0\bigl(\|\hat{X}(t) - t\hat{m} \| \geq \sigma t\bigr) ~\leq~C \exp(-\theta t). 
\ee
\end{lemma} 
\begin{proof} This is a consequence of Cramer's large deviation upper bound (see for instance Gartner-Ellis theorem in the book of Dembo and Zeitouni~\cite{D-Z}). Indeed, because of  the hypotheses (B2), for any closed set $F\subset \R^d$,
\be\label{example_lemma1_e2} 
\limsup_{t\to\infty} \frac{1}{t}\log \P\left(\frac{1}{t}\hat{X}(t) \in F\right) ~\leq~ - \inf_{v\in F} \Lambda^*(v) 
\ee
where 
\[
\Lambda^*(v) ~=~ \sup_{\alpha\in\R^d} \left(\alpha\cdot v - \Lambda(\alpha)\right) 
\]
is the convex conjugate of the function $\Lambda(\alpha) = \log\E_0(\exp(\alpha\cdot \hat{X}(1)))$. When applied with $F=\{v\in\R^d :~\|v-\hat{m}\| \geq \sigma\}$ for $\sigma>0$,  the upper large deviation  bound \eqref{example_lemma1_e2} proves \eqref{exemple_lemma1_e1} with some $C > 0$ and 
\[
\theta = \inf_{v\in \R^d : \|v-\hat{m}\| \geq \sigma}\Lambda^*(v)/2
\]
whenever 
\be\label{example_lemma1_e3} 
 \inf_{v\in \R^d : \|v-\hat{m}\| \geq \sigma} \Lambda^*(v) > 0. 
\ee
To complete the proof of Lemma~\ref{example_lemma1} it is therefore sufficient  to prove that for any $\sigma > 0$, \eqref{example_lemma1_e3}  holds. For this we notice that under the hypotheses (B2), the function $\Lambda$ is $C^\infty$ in a neighborhood of the origin $0\in\R^d$ with $\Lambda(0)= 0$ and $\nabla\Lambda(0) = \hat{m}$. By Taylor expansion, on gets therefore 
\[
L(\alpha) = \alpha\cdot \hat{m} + \frac{1}{2} \alpha\cdot \partial^2\Lambda(0)\alpha  + o(\alpha) 
\]
where $\partial^2\Lambda(0)$ denotes the Hessian matrix of $\Lambda$ at $0$, and $o(\alpha)/\|\alpha\|^2 \to 0$ when $\alpha\to 0$. This proves that for some $\delta_0 > 0$ and $C_0 >0$, 
\[
\Lambda(\alpha) \leq \alpha\cdot \hat{m}  +  C_0\|\alpha\|^2 \quad \quad  \text{whenever} \quad \quad \|\alpha\| \leq \delta_0,
\]
and consequently, for any $0 < \delta < \delta_0$ and $v\in\R^d$ with $\|v-\hat{m}\| \geq \sigma$, 
\begin{align*}
\Lambda^*(v) &= \sup_{\alpha\in\R^d} \left( \alpha\cdot v - \Lambda(\alpha)\right) ~\geq~ \sup_{\alpha\in\R^d : \|\alpha\| \leq\delta} \left( \alpha\cdot (v - \hat{m}) + \alpha\cdot \hat{m} - \Lambda(\alpha)\right) \\
&\geq   \sup_{\alpha\in\R^d : \|\alpha\| \leq\delta} \left(\alpha\cdot (v-\hat{m}) - C_0\|\alpha\|^2\right)  ~\geq~ \delta \|v-\hat{m}\| - C_0 \delta^2 ~\geq~ \delta( \sigma - C_0\delta). 
\end{align*}
Letting $\delta = \min\{\delta_0, \sigma/(2C_0)\}$ we conclude therefore that 
\[
\inf_{v\in \R^d : \|v-\hat{m}\| \geq \sigma} \Lambda^*(v)  ~\geq~ \delta( \sigma - C_0\delta) ~\geq~ \delta \sigma/2 ~>~ 0. 
\]
\end{proof}

\begin{lemma}\label{example_lemma2} For any $\delta > 0$ small enough, there is $\kappa > 0$  such that for any $x\in M\Z^d$ and $t\in\N$ 
\be\label{example_lemma2_e1} 
\P_0(\hat{X}(t) = x) \leq  \exp(-\delta \|x\|) \quad  \text{whenever }  \quad  t < \kappa \|x\|. 
\ee
\end{lemma} 
\begin{proof} The proof of this lemma is quite similar to the proof of Lemma~\ref{rl_lemma1}. Because of Assumption~$(B2)$, there is $\delta_0 > 0$ such that 
\[
C ~\dot=~ \sup_{\alpha\in\R^d~:~\|\alpha\| \leq 2\delta_0} \E_0\bigl(\exp(\alpha\cdot \hat{X}(1)) \bigr)   
~=~ \sup_{\alpha\in\R^d~:~\|\alpha\| \leq 2\delta_0} R\bigl(^t\!M\alpha\bigr) ~<~+\infty.
\]
For any $0< \delta \leq \delta_0$ and $\alpha \in\R^d$ with $\|\alpha\| \leq 2\delta$, one gets therefore 
\[
\P_0(\hat{X}(t) = x) ~\leq~ \exp(-\alpha\cdot(x)) R^t(\alpha) ~\leq~ C^t \exp(-\alpha\cdot(y-x))  
\]
When applied with $\alpha = 2\delta (x)/\|x\|$, the last relation proves that 
\[
\P_0(\hat{X}(t) = x) ~\leq~ \exp\bigl(- 2\delta\|x\| + t \ln C\bigr) 
\]
and consequently, letting $\kappa = \delta/\ln C$ one gets \eqref{example_lemma2_e1}. 
\end{proof} 
When combined together, Lemma~\ref{example_lemma1} and Lemma~\ref{example_lemma2} imply the following statement.

\begin{lemma}\label{example_lemma3} For any $\sigma > 0$ there are two strictly positive real numbers $\delta > 0$ and $C > 0$ such that for any $t\in\N$ and $x\in M\Z^d$,
\be\label{example_lemma3_e1} 
\P_0(\hat{X}(t)=x) \leq C \exp( -\delta t - \delta \|x\| ) \quad \text{whenever} \quad \|x - t \hat{m}\| \geq \sigma t. 
\ee
\end{lemma} 
\begin{proof} Indeed, by Lemma~\ref{example_lemma2}, for any $\delta_1 > 0$ small enough, there is $\kappa > 0$ such that for any $x\in M\Z^d$ and $t\in\N$ satisfying the inequality $t < \kappa \|x\|$, the following relation holds
\begin{align}
\P_0(\hat{X}(t)=x) &\leq \exp(- 2\delta_1 \|x\|)  \nonumber \\
&\leq~ \exp\left( -\frac{\delta_1}{\kappa} t - \delta_1 \|x\| \right). \label{example_lemma3_e2} 
\end{align}
By Lemma~\ref{example_lemma1}, for any $\sigma > 0$ there are $C> 0$ and $\theta > 0$ such that for any $x\in M\Z^d$ and $t\in\N$, 
\[
\P_0(\hat{X}(t)=x) \leq C \exp( - 2\theta t ) \quad \text{whenever} \quad \|x - t \hat{m} \| \geq \sigma t.  
\]
If $t \geq \kappa \|x\|$ and  $\|x - t \hat{m}\| \geq \sigma t$   one gets therefore 
\[
\P_0(\hat{X}(t)=x) \leq C \exp( - \theta t  -  \theta\kappa\|x\|).
\]
When combined with \eqref{example_lemma3_e2} the last inequality proves \eqref{example_lemma3_e1} with $\delta = \min\{\delta_1, \delta_1/\kappa, \theta, \theta\kappa\}$. 
\end{proof} 
As a consequence of Lemma~\ref{example_lemma3} we obtain the following statement.

\begin{lemma}\label{example_lemma3a} For any $\eps > 0$ there are two strictly positive real numbers $\delta > 0$ and $C > 0$ such that for any $t\in\N$ and $x\in M\Z^d\cap{\cal C}_\eps$ and $y\in M\Z^d\setminus{\cal C}_\eps$,
\be\label{example_lemma3a_e1} 
\P_x(\hat{X}(t)=y) \leq C \exp( -\delta t - \delta \|y-x\| ) 
\ee
\end{lemma} 
\begin{proof} Indeed,  suppose first that $x=0$. For any $t\in\N$, the distance between the point $t \hat{m}$ and the boundary $\partial{\cal C}_\eps = \{y\in\R^d~:~ \angle(x, \hat{m}) = \eps\}$ of the cone ${\cal C}_\eps$ is equal to $t \|\hat{m}\|\, t \sin(\eps)$, and consequently,  for any $z\in M\Z^d\setminus{\cal C}_\eps$, 
\[
\|z - t \hat{m} \|  ~\geq~ t \|\hat{m}\| \sin(\eps).
\]
By Lemma~\ref{example_lemma3} applied with $\sigma = \|\hat{m}\| \sin(\eps)$, there are therefore two strictly positive real numbers $\delta > 0$ and $C > 0$ such that for any $t\in\N$ and $z\in M\Z^d\setminus{\cal C}_\eps$, 
\be\label{example_lemma3a_e2} 
\P_0(\hat{X}(t)=z) \leq C \exp( -\delta t - \delta \|z\| ) 
\ee
For $x=0$, our lemma is therefore proved. To prove \eqref{example_lemma3a_e1} for an arbitrary $x\in M\Z^d\cap{\cal C}_\eps$ it is sufficient to notice that for any $x\in M\Z^d\cap{\cal C}_\eps$, 
\[
 x+{\cal C}_\eps \subset {\cal C}_\eps,
\]
for any $x,y\in M\Z^d$, 
\[
\P_x(\hat{X}(t)=y) = \P_0(\hat{X}(t) = y-x), 
\] 
and for  $y\in M\Z^d\setminus{\cal C}_\eps$, 
\[
y-x \in M\Z^d\setminus( x + {\cal C}_\eps).
\]
Using therefore \eqref{example_lemma3a_e2} with $z=y-x$ , one gets 
\eqref{example_lemma3a_e1}. 
\end{proof} 

Now we are able to get 
\begin{lemma}\label{example_lemma4}  Let a sequence of points $(x_n)\subset \left({\cal C}_\eps\cap M\Z^d\right)^\N$ be such that 
\[
\lim_{n\to\infty} x_n\cdot m ~=~+ \infty \quad \text{and} \quad \sup_{n} \|Pr(x_n)\| ~<~+\infty. 
\]
Then 
\be\label{e6} 
\lim_{n\to\infty} \E_{x_n} \left({\cal V}\left(\hat{X}\bigl(\tau_\eps\wedge\tau_{{\cal K}(\theta_2,\hat{v})}\bigr)\right)\right)  ~=~0.
\ee
\end{lemma} 
\begin{proof} 
Recall  that ${\cal V}(y) ~=~ 0$ for any $y\in\R^d\setminus{\cal C}(\theta_2,\hat{v})$. Hence,  on the event $\tau_{{\cal K}(\theta_2,\hat{v})}\leq \tau_\eps$,
\[
{\cal V}\left(\hat{X}\bigl(\tau_\eps\wedge\tau_{{\cal K}(\theta_2,\hat{v})}\bigr)\right) ~=~ 0,
\]
and using \eqref{e3p},
\begin{align*}
\E_{x_n} \left({\cal V}\left(\hat{X}\bigl(\tau_\eps\wedge\tau_{{\cal K}(\theta_2,\hat{v})}\bigr)\right)\right) &= \E_x \left({\cal V}\left(\hat{X}\bigl(\tau_\eps\bigr)\right); \;  \tau_{{\cal K}(\theta_2,\hat{v})} > \tau_\eps\right) \\
&\leq~ C\, \E_{x_n} \left((\|Pr(\hat{X}\bigl(\tau_\eps\bigr))\|^2 + 1); \;  \tau_{{\cal K}(\theta_2,\hat{v})} > \tau_\eps\right) \\
&\leq~ C\, \E_{x_n} \left((\|Pr(\hat{X}\bigl(\tau_\eps\bigr))\|^2+1); \; \tau_\eps < +\infty\right) \\ 
\end{align*} 
where by Lemma~\ref{example_lemma3a}, 
\begin{align*}
\E_{x_n} \!\Bigl((\|Pr(\hat{X}\bigl(\tau_\eps\bigr))\|^2+1); \, &\tau_\eps < +\infty\Bigr) ~=~ \!\!\!\sum_{y\in M\Z^d\setminus{\cal C}_\eps} \sum_{t=1}^\infty (\|Pr(y)\|^2 +1) \P_{x_n}\!(\hat{X}\bigl(t\bigr) = y,  \tau_\eps = t) \\
&\leq~ \sum_{y\in M\Z^d\setminus{\cal C}_\eps} \sum_{t=1}^\infty (\|Pr(y)\|^2 +1) \P_{x_n}(\hat{X}(t) ~=~ y) \\
&\leq~ C \sum_{y\in M\Z^d\setminus{\cal C}_\eps} ~\sum_{t=1}^\infty (\|Pr(y)\|^2 +1)\exp( -\delta t - \delta \|y-x_n\| ) \\
&\leq~ C (1- e^{-\delta})^{-1} \sum_{y\in M\Z^d\setminus{\cal C}_\eps} (\|Pr(y)\|^2 +1)\exp( - \delta \|y-x_n\| )
\end{align*}  
Using the inequalities 
\[
 \|y-x_n\| ~\geq~ \|Pr(y) - Pr(x_n)\| ~\geq~ \|Pr(y)\| - \|Pr(x_n)\|,
\]
\[
\sup_{y\in\R^d} \left((\|Pr(y)\| +1) \exp\left(-\frac{\delta}{3} \|Pr(y)\|\right)\right) ~\leq~ \sup_{\ \gamma > 0} (1+\gamma)e^{\delta\gamma/3} ~\leq~ 3/\delta, 
\]
and  
\[
\sup_{n} \|Pr(x_n)\| ~<~+\infty, 
\]
we obtain therefore 
\begin{align*}
\E_{x_n} \bigl((\|Pr(\hat{X}\bigl(\tau_\eps\bigr))\|^2+1); \, \tau_\eps < +\infty \bigr)  \leq C' \sum_{y\in M\Z^d\setminus{\cal C}_\eps} \exp\left( - \frac{2}{3} \delta \|y-x_n\|\|\right)
\end{align*} 
with some $C' > 0$. Since 
\[
\sup_{y\in \R^d\setminus{\cal C}_\eps} \exp\left( - \frac{\delta}{3} \|y-x_n\| \right) ~=~ \exp\left( - \frac{\delta}{3} \dist(x_n,\partial{\cal C}_\eps) \right),
\] 
and
\[
\sum_{y\in M\Z^d} \exp\left( - \frac{\delta}{3} \|y-x_n\| \right) ~=~ \sum_{y\in M\Z^d} \exp\left( - \frac{\delta}{3} \|y\| \right) ~<~+ \infty, 
\]
this proves that 
\[
\E_{x_n} \left((\|Pr(\hat{X}\bigl(\tau_\eps\bigr))\|^2+1); \, \tau_\eps < +\infty\right)  ~\leq~ C'  \exp\left( - \frac{\delta}{3} \dist(x_n,\partial{\cal C}_\eps )\right) 
\]
with some $C'' > 0$. Since under the hypotheses of Lemma~\ref{example_lemma4}, $\lim_n \dist(x_n,\partial{\cal C}_\eps) ~=~+ \infty$, le last inequality proves \eqref{e6}. 
\end{proof} 
Now we are ready to complete the proof of Theorem~\ref{main_result_example_th1}. 

\noindent
By Proposition~\ref{Denisov_Wachtel} there is a point $x_0\in{\cal K}(\theta_2,\hat{v})$ such that 
\[
\inf_{x\in x_0 + {\cal K}(\theta_2,\hat{v})} {\cal V}(x) \geq {\cal V}(x_0) > 0.
\]
Choose a sequence $(x_n)\in \left({\cal C}_\eps\cap M\Z^d\right)^\N$ such that for any $n\geq 1$, $Pr(x_n) \in x_0+{\cal K}(\theta_2,\hat{v})$, 
\[
\sup_{n} \|Pr(x_n)\| ~<~+\infty  \quad \text{and} \quad \lim_{n\to\infty} x_n\cdot m ~=~+ \infty. 
\]
Then according to the definition of the function ${\cal V}:{\cal C}(\theta_2,\hat{v})\to\R_+$, 
\[
\inf_n {\cal V}(x_n) ~= ~\inf_n {\cal V}(Pr(x_n)) ~\geq~ {\cal V}(x_0) > 0, 
\]
and consequently, by Lemma~\ref{example_lemma3}, for any $\eps > 0$ ans $n\in\N$ large enough,
\[
{\cal V}(x_n)  -   \E_{x_n} \left({\cal V}\left(\hat{X}\bigl(\tau_\eps\wedge\tau_{{\cal K}(\theta_2,\hat{v})}\bigr)\right)\right) ~>~ 0.
\]
By Proposition~\ref{prop5_2} applied with $S(t) = Pr(X(t))$ and $\tau^*=\tau_\eps$, from this it follows that for any $\eps > 0$ ans $n\in\N$ large enough,
\be\label{e7}
\E_{x_n}(\tau_\eps\wedge\tau_{\cal K}(\theta_2,\hat{v})) ~=~ +\infty. 
\ee
Since the boundary of the cone ${\cal C}$  is $C^1$, according to the definition of the cone ${\cal C}(\theta_2,\hat{v})$, there is $\eps > 0$ such that 
\[
{\cal C}_\eps \cap {\cal C}(\theta_2,\hat{v}) ~\subset~{\cal C}, 
\]
and consequently, for any $x\in M\Z^d \cap {\cal C}_\eps \cap {\cal C}(\theta_2,\hat{v})$, $\P_x$-a.s. 
$
\hat\tau ~\geq~ \tau_\eps\wedge\tau_{{\cal K}(\theta_2,\hat{v})}$. Using \ref{e7} from this it follows that for $n$ large enough, $\E_{x_n}(\hat\tau) = +\infty$, and since under the hypotheses our theorem, the random walk $(\hat{X}(t))$ is irreducible in $\hat{E}={\cal C}\cap M\Z^d$, this proves that $\E_x(\hat\tau) = + \infty$ for all $x\in\hat{E}$. Theorem~\ref{main_result_example_th1} is therefore proved.

\section{Proof of Theorem~\ref{main_result_example_th2}.}\label{th2_example_proof} 
Under the hypotheses of Theorem~\ref{main_result_example_th2}, for any $\alpha\in\partial_+ D$ the twisted random walk $(X_\alpha(t))$ with transition probabilities $p_\alpha(x,y) = \exp(\alpha\cdot (y-x)) \mu(y-x)$ satisfies the conditions (B0)-(B3), and consequently,   by Theorem~\ref{main_result_example_th1},  for any $x\in E = {\cal C}\cap \Z^d$, $\E_x(\tau_\alpha) = + \infty$. According to the definition of the boundary set $\partial_+^\infty S$, from this it follows that $\partial_+^\infty S = \partial_+ S$. When combined with Theorem~\ref{theorem_4}, this result proves that  for any $q\in S_+$, the function  $k_q$  is a  finite,  non-zero  and harmonic  for $(Z(t))$, and that for any sequence of points $(y_n)\in({E})^\N$ with $\lim_n\|y_n\| = \infty$ and $\lim_n y_n/\|y_n\| = q$, \, \eqref{theorem_4_e1} holds.

\bibliographystyle{amsplain}

\providecommand{\bysame}{\leavevmode\hbox to3em{\hrulefill}\thinspace}
\providecommand{\MR}{\relax\ifhmode\unskip\space\fi MR }
\providecommand{\MRhref}[2]{%
  \href{http://www.ams.org/mathscinet-getitem?mr=#1}{#2}
}
\providecommand{\href}[2]{#2}

\end{document}